\documentclass[10pt]{article}

\usepackage{pdfsync}

\usepackage{datetime2} 

\usepackage[utf8]{inputenc}
\usepackage{amsmath}
\usepackage{amsfonts}
\usepackage{amssymb}
\usepackage{yfonts} 
\usepackage{eufrak}
\usepackage{mathrsfs}
\usepackage{dsfont}
\usepackage{fancyhdr}
\usepackage{indentfirst}
\usepackage{graphicx}
\usepackage{newlfont}
\usepackage{amssymb}
\usepackage{amsmath}
\usepackage{latexsym}
\usepackage{amsthm}
\usepackage{mathtools}
\usepackage{nicefrac}
\usepackage{epstopdf}
\usepackage{caption}
\usepackage{color}
\usepackage{dsfont}
\usepackage{comment}
\usepackage{hyperref}
\usepackage{url}
\usepackage{esvect}

\usepackage{amsmath}                                                        
\numberwithin{equation}{section}
\usepackage{graphicx}                                                       
\usepackage{subfigure}
\usepackage{enumitem}                                                    
\usepackage{hyperref}
\usepackage{verbatim}
\usepackage{xcolor}

\usepackage{amssymb}                                                        
\usepackage[mathscr]{eucal}                                             
\usepackage{cancel}                                                             
\usepackage[normalem]{ulem}                                                                 
\usepackage{pstricks}
\usepackage{rotating}
\usepackage{lscape}
\usepackage[paperwidth=8.5in,paperheight=11in,top=1.00in, bottom=1.00in, left=1.00in, right=1.00in]{geometry}
\usepackage{mathtools}                                                      
\mathtoolsset{showonlyrefs=true}                                    
\usepackage{fixltx2e,amsmath}                                           
\MakeRobust{\eqref}

\usepackage{dsfont}

\usepackage{mathdots}
\usepackage{amsthm}                                                             
\allowdisplaybreaks

\theoremstyle{plain}
\newtheorem{theorem}{Theorem}
\numberwithin{theorem}{section}

\theoremstyle{definition}
\newtheorem{definition}[theorem]{Definition}

\newtheorem{notation}[theorem]{Notation}
\newtheorem{remark}[theorem]{Remark}
\newtheorem{assumption}[theorem]{Assumption}

\makeatletter
\DeclareRobustCommand{\cev}[1]{%
  {\mathpalette\do@cev{#1}}%
}
\newcommand{\do@cev}[2]{%
  \vbox{\offinterlineskip
    \sbox\z@{$\m@th#1 x$}%
    \ialign{##\cr
      \hidewidth\reflectbox{$\m@th#1\vec{}\mkern4mu$}\hidewidth\cr
      \noalign{\kern-\ht\z@}
      $\m@th#1#2$\cr
    }%
  }%
} \makeatother

\usepackage{bm}
\def \VX {\mathbf{Z}}
\def \BZ {\z}
\def \zz {{z}}
\def \bz {{\z}}
\def \KK {\mathcal{L}_{\m}}
\def \ff {{F}}
\def \gg {{G}}

\def \y {{\eta}}
\def \s {{\sigma}}

\def \a {{\alpha}}
\def \b {{\beta}}

\def \R {\mathbb{R}}
\def \p {\partial}

\renewcommand{\[}{\left[}

\renewcommand\d{\delta}

\def \Lip {\textnormal{Lip}}

\def \R  {{\mathbb {R}}}

\def \n {{\nu}}
\def \m {{\mu}}
\def \y {{\eta}}
\def \th {{\theta}}
\def \z {{\zeta}}
\def \p {{\partial}}
\def \a {{\alpha}}
\def \O {{\Omega}}

\def \d {{\delta}}

\def \var {{\text{var}}}

\def \k {{\kappa}}

\def \a {{\alpha}}
\def \b {{\beta}}
\def \d {{\delta}}

\def \G {\Ga}

\def \Ga {{\Gamma}}

\def \It\^o {It\^o }
\def \s {{\sigma}}
\def \K {\mathcal{L}}
\def \Y {\mathcal{Y}}

\def \R {{\mathbb {R}}}

\def \C {\mathbf{C}}

\def \n {{\nu}}

\def \m {{\mu}}
\def \y {{\eta}}
\def \th {{\theta}}

\def \O {{\Omega}}
\def \phi {{\varphi}}

\def \tilde {\widetilde}

\def\l {\lambda}

\def \A {\mathcal{A}}

\def \Ã  {{\`a }}
\def \Ã¨ {{\`e }}
\def \Ã² {{\`o }}
\def \Ã¹ {{\`u }}

\begin{document}

\title{Trading in residential energy systems with storage:\\ a kinetic mean-field approach}

\author{
Margherita Fabini \and Andrea Pascucci\thanks{Dipartimento di Matematica, Universit\`a di Bologna,
Bologna, Italy. \textbf{e-mail}: andrea.pascucci@unibo.it} \and Alessio
Rondelli\thanks{Dipartimento di Matematica, Universit\`a di Bologna, Bologna, Italy.
\textbf{e-mail}: alessio.rondelli2@unibo.it}}

\date{This version: \today}
\maketitle

\begin{abstract}
We study a stochastic optimal control problem motivated by the operation of a large ensemble of
residential storage devices coordinated by an energy aggregator. The aggregator remunerates
prosumers in exchange for direct control of their batteries and seeks to jointly (i) reduce local
supply-demand imbalances and (ii) exploit intraday price fluctuations through energy arbitrage.
The core modeling feature is a {\it kinetic mean-field formulation}: the state of charge is
treated as a position, the charging/discharging power as a velocity, and the control as an
acceleration, thus encoding ramp-rate limitations and producing smooth power trajectories. This
leads to a controlled McKean-Vlasov Langevin-type system in which both the drift and the objective
functional depend on the time-marginal law of the state, allowing one to capture endogenous
interaction effects and population-level stabilization incentives. The performance criterion
combines the cost of grid exchange with convex penalties representing degradation and control
effort, and includes mean-field terms that promote alignment with the population average; terminal
contributions account for residual energy value and end-of-horizon coordination. The resulting
control problem is Markovian and hypoelliptic, and naturally connects mean-field control with
ultraparabolic operators of kinetic type. This viewpoint provides a coherent bridge between
physically constrained storage actuation and law-dependent incentives in large-scale energy
management. Numerical experiments based on deep learning solvers are presented to validate the
model. From a computational standpoint, the problem is particularly challenging, as it yields a
fully coupled forward-backward stochastic system associated with a five-dimensional
Hamilton-Jacobi-Bellman equation.
\end{abstract}



\section*{Declarations}

\subsection*{Availability of data and material}
The electricity price data (Italian day-ahead PUN) are publicly available from the website of the
Gestore dei Mercati Energetici (GME). The household photovoltaic and load time series used for
calibration are sourced from the publicly available SolarEdge Monitoring platform (see Section
3.1--3.2 of the manuscript). Additional materials (e.g., processed datasets and implementation
details sufficient to reproduce the numerical experiments) are available from the corresponding
author upon reasonable request.

\subsection*{Competing interests}
The authors declare that they have no competing interests.

\subsection*{Funding}
The authors declare that no specific funding was received for this work.

\subsection*{Authors' contributions}
M.F., A.P., and A.R. jointly conceived the study and developed the methodology. All authors
contributed to the mathematical analysis and to the interpretation of the results. Numerical
experiments were implemented and validated by the authors. All authors wrote the manuscript and
approved the final version.

\subsection*{Acknowledgements}
A.P. and A.R. are members of INDAM-GNAMPA. The authors are grateful to Vincenzo Vespri for
inviting us to contribute to this special issue and for bringing to our attention the
mathematically interesting problems that arise in energy systems with storage. We thank Daniele
Spinella for several stimulating discussions on the topic. We also thank Emanuele Lippi for his
assistance with the numerical implementation of the tests.

\section{Introduction}\label{intro}
We investigate a stochastic optimal control problem modeling the strategy of an energy aggregator
managing a distributed network of grid-connected residential prosumers equipped with photovoltaic
panels and battery storage. The framework is grounded in a contractual arrangement wherein the
aggregator directly controls the storage assets in exchange for prosumer remuneration. This
centralized coordination targets two complementary objectives: (i) enhancing grid reliability by
mitigating supply-demand imbalances and smoothing local fluctuations; and (ii) extracting value
through energy arbitrage by leveraging electricity price volatility. The arbitrage component is
particularly relevant when storage is structurally underutilized, a scenario typical of winter
months characterized by reduced solar irradiance. Unlike individual households, the aggregator can
access wholesale markets to monetize intraday price spreads, effectively transforming otherwise
idle storage capacity into a revenue-generating asset. This aggregator-based modeling is
explicitly motivated by recent U.S. regulatory developments, most notably FERC Order No.\ 841,
which requires organized wholesale electricity markets to admit energy storage resources,
including behind-the-meter and aggregated residential batteries, under market rules tailored to
their operational constraints; see \cite{Mich} for background and discussion. This regulatory
shift makes the centralized stochastic optimal control of a large population of grid-connected
prosumers a natural and economically meaningful problem.

To translate these economic and regulatory considerations into a tractable mathematical problem,
we formulate the aggregator's decision-making process as a continuous-time stochastic control
problem for a large population of interacting systems. The resulting dynamics naturally combine
idiosyncratic uncertainty at the household level with aggregate effects induced by coordinated
operation, leading to a controlled mean-field formulation. More precisely, we consider a stylized
model governed by the following state processes:
\begin{itemize}
  \item $S_t$: electricity price;
  \item $H_t$: net home load (in kW), i.e., household consumption minus renewable generation;
  \item $V_t^{a}$: battery power (in kW), with the convention $V_t^{a}>0$ for charging
        and $V_t^{a}<0$ for discharging;
  \item $X^{a}_{t}$: battery's state of charge (SOC), defined as
     $$X^{a}_{t}=X^{a}_{0}+\int_{0}^{t}V^{a}_{s}d s;$$
    \item $a_{t}$: control process representing the rate of change of power (ramp rate).
\end{itemize}
We denote by
  $$\VX_{t}^{a}:=(S_{t}, H_{t}, V_{t}^{a}, X_{t}^{a}),\qquad t\in[0,T],$$
the state process and by $[\VX_{t}^{a}]$ its time-marginal distribution. 
We assume that $\VX^{a}$ solves the {\it controlled McKean-Vlasov} system
 $$
\begin{cases}
  dS_t = b^{S}(t,S_{t}) dt + \sigma^{S}(t,S_{t}) dW^{S}_t,\\
  dH_t = b^{H}(t,H_{t}) dt + \sigma^{H}(t,H_{t}) dW^{H}_t,\\
  dV^{a}_{t}=b^{V}(t,\VX^{a}_{t},[\VX_{t}^{a}],a_{t})dt+\s^V(t,\VX^{a}_{t},[\VX_{t}^{a}])dW^{V}_{t},\\
  dX^{a}_{t}=V^{a}_{t}dt,
\end{cases}
 $$
where $(W^S,W^H,W^V)$ is a three-dimensional Brownian motion. The energy manager aims to minimize
a cost functional of the form
\begin{equation}\label{e2}
  J(a) = E\left[ \int_0^T f(t, \VX^{a}_t,[\VX_{t}^{a}],a_t) dt +
  g(\VX^{a}_T,[\VX_{T}^{a}])\right],
\end{equation}
representing the trade-off between energy costs, hardware degradation, and control effort.

\subsection{A kinetic mean-field approach}
Prior to specifying the precise structure of the model coefficients and the objective functional,
we outline the mathematical formulation of the problem. A distinguishing feature of our approach
is the adoption of a \emph{kinetic mean-field} framework. Specifically, the system dynamics are
governed by 
the controlled Langevin-type equation
\begin{equation}\label{e1}
 \begin{cases}
  dV^{a}_{t}=b(t,\VX^{a}_{t},[\VX^{a}_{t}],a_{t})dt+\s(t,\VX^{a}_{t},[\VX^{a}_{t}])dW_{t},\\
  dX^{a}_{t}=V^{a}_{t}dt.
 \end{cases}
\end{equation}
Langevin dynamics of the form \eqref{e1} are ubiquitous in physics, biology, and mathematical
finance, having been extensively studied since the foundational works \cite{MR1503147,Langevin}.
They appear naturally in the kinetic theory of gases and fluid dynamics (see, e.g.,
\cite{MR2459454,MR8130}). From a financial perspective, \eqref{e1} can be viewed as a nonlinear
mean-field generalization of the degenerate SDEs typically encountered in the valuation of
path-dependent derivatives, such as Asian options; see \cite{MR1830951,MR2791231}.

In our model the system state is represented through a {\it phase-space analogy} in which:
\begin{itemize}
    \item the battery state of charge $X^{a}$ plays the role of the \textit{position};
    \item the charging/discharging power $V^{a}$ plays the role of the \textit{velocity};
    \item the control $a$ plays the role of the \textit{acceleration} (or force).
\end{itemize}
Unlike standard models in which power can be adjusted instantaneously, our formulation explicitly
accounts for physical ramp-rate constraints, yielding a smooth evolution of the power output.
Moreover, \eqref{e1} features \textit{mean-field} terms: both the dynamics and the cost depend not
only on the individual state process but also on its time-marginal law. This dependence enables a
consistent description of the collective behavior of a large population of heterogeneous storage
devices under centralized control.

\subsection{Related literature and numerical aspects}
The optimal management of distributed energy resources has recently attracted considerable
attention across engineering, economics, and applied mathematics, and raises modeling and
mathematical challenges due to physical constraints, market frictions, and large-scale
interactions. Existing contributions adopt complementary perspectives, ranging from decentralized
competitive formulations to centralized coordination mechanisms.

A significant strand of literature focuses on decentralized approaches, often modeled through Mean
Field Games (MFG), where individual agents optimize their strategies against an aggregate signal.
In a purely deterministic setting, \cite{Strbac} investigates the coordination of micro-storage
devices where the market-clearing price is characterized as the fixed point of a coupled
Hamilton-Jacobi-Bellman system, broadcast to agents implementing decentralized responses. Moving
to the stochastic framework, \cite{MR4054298} models a large smart grid as a nonzero-sum
$N$-player game with endogenous price formation. In their setting, each node buys or sells
electricity based on instantaneous net home load and storage decisions, directly impacting the
spot price. The authors derive the corresponding extended MFG limit with common noise and compare
the decentralized equilibrium with the solution of a central planner, obtaining explicit results
in a linear-quadratic specification.

Complementary to these game-theoretic formulations are centralized control approaches, where a
single entity (e.g., an aggregator) optimizes the system's global performance. This perspective,
adopted in our work, often requires advanced techniques to handle constraints and dimensionality.
For instance, \cite{MR4255454} addresses the optimal operation of pumped-hydro storage with hard
capacity constraints via a level-set reformulation, though in a single-agent setting {\it without
mean-field interactions}. In the context of large populations, \cite{MR4445582} develops a
Pontryagin-type analysis for an extended McKean-Vlasov control problem applied to smart grids.
Their work accommodates non-quadratic costs and general filtrations, allowing for jump-type
shocks, thus departing from the standard Markovian linear-quadratic assumptions. Similarly,
\cite{MR4588432} proposes a level-set methodology for McKean-Vlasov control under running and
terminal constraints in law. In their electricity storage application, the spot price features a
mean-field price-impact term driven by aggregate net injections, while the controller acts through
constrained injection-withdrawal decisions. One of the differences of our approach lies in the
\emph{kinetic} nature of the control problem. Unlike standard formulations that control power
directly and may yield discontinuous strategies, our model controls the ramp rate, ensuring smooth
power profiles that respect physical battery constraints.

We recall that well-posedness results for Langevin backward SDEs (in the non-McKean-Vlasov
setting) were first established in \cite{MR1941093}. Since then, this class of degenerate SDEs and
associated PDEs has been investigated in several works (see, e.g., \cite{MR2659772},
\cite{MR2802040}, \cite{MR4660246}, and the references therein). The extension to the
McKean-Vlasov framework has been addressed more recently in \cite{MR4612657}, in the infinite
dimensional case, and \cite{pasron2025}.

From a mathematical perspective, the kinetic formulation preserves the Markovian structure of the
problem, allowing us to exploit the extensive theory and recent advances concerning hypoelliptic
operators of ultraparabolic type (see, e.g., \cite{MR3923847}, \cite{MR4700191},
\cite{MR4069607}). At the same time, the present framework raises a number of new questions that
are interesting in their own right and deserve further investigation. In order to keep the
exposition focused on the main contributions, we only sketch the proof of some results (e.g.
Theorem \ref{t1}). In particular, we outline the main steps of the argument and explain how the
kinetic formulation can be combined with the techniques in \cite{MR2048401} to obtain the result.
Similarly, the proof of Theorem \ref{th3} relies on classical continuous-dependence results for
PDEs, extended here to the ultraparabolic setting following the approach in \cite{MR2352998}. A
comprehensive mathematical treatment, including all intermediate technical details, is provided in
the companion work \cite{pas_ron}.

From the computational standpoint, the McKean-Vlasov optimality system leads to a coupled
Forward-Backward SDE in which the backward component depends on the law of the forward state,
making classical grid-based schemes impractical. Numerical methods for this class of problems are
still developing (see, e.g., \cite{MR3736669,MR4522347}). In our numerical tests we therefore use
deep learning-based solvers, following \cite{MR4478485} and \cite{Raissi_2018}, which propose a neural-network
approximation of the decoupling fields, together with dynamically updated moment estimates to treat the mean-field terms.

\medskip
\noindent {\bf Structure of the paper.} This paper is organized as follows. In Section~\ref{sec2}
we present the main results together with the standing assumptions. In particular, we establish
the connection between the stochastic optimal control problem, the associated forward-backward
stochastic differential equations, and the Hamilton-Jacobi-Bellman equation driven by a kinetic
hypoelliptic operator. We also introduce the anisotropic and intrinsic functional spaces that
provide the natural analytical framework for the problem. Section~\ref{sec:numerical-tests} is devoted to
numerical experiments: we specify a fully consistent stylized model and describe the data used for
calibration as well as for the numerical tests. Finally, Section~\ref{sec4} contains the proof of
the main result, Theorem~\ref{th3}.

\medskip
\noindent Throughout the paper we use the following general notation.
\begin{itemize}
  \item We write $f\lesssim g$ if there exists a constant $\k>0$ such that
  $$f\le \k g.$$
  The constant $\k$ is \emph{structural}: it may depend only on the fixed data of the problem,
  in particular the dimension, the time horizon $T$, and bounds on the coefficients appearing
  in the main assumptions, but it is independent of the relevant variables and of the specific
  functions $f$ and $g$ involved in the estimate.

  \item For $t\in[0,T]$, $\Lip_{t,T}(\R^{d})$ denotes the set of real-valued functions $u$ defined on
  $[t,T]\times\R^{d}$ such that $u(s,\cdot)$ is Lipschitz continuous on $\R^{d}$ uniformly with
  respect to $s\in[t,T]$, namely
  $$\|u\|_{\Lip_{t,T}}
    :=\sup_{t\le s\le T}\sup_{z,z'\in \R^{d}\atop z\neq z'}
      \frac{|u(s,z)-u(s,z')|}{|z-z'|}<\infty.$$

  \item $\mathcal{P}_p(\R^{d})$ denotes the set of probability measures on $\R^{d}$ with finite
  $p$-th moment, endowed with the $p$-Wasserstein distance
  $$\mathcal{W}_{p}(\mu, \nu)
    :=\left(\inf_{\pi \in \Pi(\mu, \nu)}
      \int_{\R^{d}\times\R^{d}} |x-y|^{p}d\pi(x,y)\right)^{1/p},$$
  where $\Pi(\mu, \nu)$ denotes the set of probability measures on $\R^{d}\times\R^{d}$ with
  marginals $\mu$ and $\nu$.
\end{itemize}

\section{General strategy and main results}\label{sec2} For ease of exposition, we carry out the rigorous
mathematical analysis only for the reduced-form model \eqref{e1}. This entails no loss of
generality, since the inclusion of additional non-degenerate processes can be treated in the same
way and does not introduce any further mathematical difficulties. In this section we outline the
general strategy and state the main results; all proofs and more technical arguments are deferred
to Section \ref{sec4}.

Let $(\Omega,\mathcal{F},P)$ be a probability space endowed with a filtration
$(\mathcal{F}_t)_{t\in[0,T]}$ satisfying the usual conditions and supporting a one-dimensional
Brownian motion $W$. We denote by $\mathbb{A}$ the set of square-integrable, progressively
measurable control processes $a=(a_t)_{0\le t\le T}$ taking values in a closed convex subset
$A\subseteq\R$. We consider the McKean-Vlasov optimal control problem
\begin{align}\label{MKVCP}
  &\begin{cases}
    dV_t^a = b(t,\VX_t^a,[\VX_t^a],a_t)dt
    + \sigma(t,\VX_t^a,[\VX_t^a])dW_t,
    \qquad \VX^a \equiv (V^a,X^a),\\
    dX_t^a = V_t^a dt,\\
    \VX_0^a = \eta \in L^2(\Omega,\mathcal{F}_0,P),
  \end{cases}\\
  &\inf_{a\in\mathbb{A}} J(a), \qquad
  J(a)=E\left[\int_0^T f(t,\VX_t^a,[\VX_t^a],a_t)dt
  + g(\VX_T^a,[\VX_T^a])\right].
\end{align}
The main result of this section is Theorem~\ref{t1}, which states that, under
Assumptions~\ref{ass1} and~\ref{ass2}, the control problem~\eqref{MKVCP} is well posed. Moreover,
the optimal control is characterized through a fully coupled McKean-Vlasov FBSDE and this
representation provides a natural framework for the numerical approximation of both the optimal
control and the associated value.

We solve the control problem \eqref{MKVCP} by means of a fixed point argument on the flow of
marginal laws. The procedure can be summarized as follows.
\begin{itemize}
\item
Given a flow of probability measures\footnote{$C([0,T],\mathcal{P}_2(\R^2))$ denotes the space of
continuous flows of probability measures on $\R^2$ with finite second moment.} 
$\m=(\mu_t)_{t\in[0,T]}\in C([0,T],\mathcal{P}_2(\R^2))$, we consider the classical stochastic
control problem obtained by freezing the measure argument in the coefficients of the McKean-Vlasov
dynamics \eqref{MKVCP}
\begin{align}\label{frozen_SOC}
  &\begin{cases}
    dV_t^{a,\m} = b(t,\VX_t^{a,\m},\m_t,a_t)dt + \s(t,\VX_t^{a,\m},\m_t)dW_t,\\
    dX_t^{a,\m} = V_t^{a,\m} dt,\\
    \VX^{a,\m}_0 = \eta\in L^2(\O,\mathcal{F}_0, P),
  \end{cases}\\
    &\inf_{a\in\mathbb{A}}J^\m(a),\quad\mathrm{with}\ J^\m(a)=E\left[\int_0^T f(t,\VX^{a,\m}_t,\m_t,a_t)dt +
    g(\VX^{a,\m}_T,\m_T)\right].
\end{align}
We prove well-posedness of \eqref{frozen_SOC} and denote by $\VX^{\m}$ the optimal state process.
\item We consider the mapping on $C([0,T],\mathcal{P}_2(\R^2))$ defined by
  $$\Phi:\m \longmapsto \left([\VX^{\mu}_t]\right)_{t\in[0,T]}$$
and show that $\Phi$ is a contraction with respect to a suitable metric. Hence, it admits a unique
fixed point, $\bar{\m}=\Phi(\bar{\m})$, which yields the solution to the original control problem
\eqref{MKVCP}.
\end{itemize}
Since the {\it frozen}\footnote{Frozen refers to the fact that the measure argument in the
McKean-Vlasov coefficients is kept fixed.} problem \eqref{frozen_SOC} no longer depends on the law
of the solution, we can construct the corresponding FBSDE and the associated HJB equation.
Moreover, since the control does not enter the diffusion coefficient $\sigma$, it is convenient to
work with the {\it reduced} Hamiltonian, defined by
\begin{equation}\label{H}
 H(t,\zz,\mu,p_v,a):= b(t,\zz,\mu,a) p_v + f(t,\zz,\mu,a),
\end{equation}
where $\zz = (v,x)$ and $p_v$ denotes the adjoint variable associated with the velocity component.
\begin{assumption}\label{ass2}
For all $(t,\zz,\m,p_{v})\in [0,T]\times \R^{2}\times\mathcal{P}_2(\R^{2})\times \R$, the reduced
Hamiltonian $H$ in \eqref{H} admits a unique minimizer $\hat{a}=\hat{a}(t,\zz,\m,p_{v})$ such that
\begin{align}\label{w11}
  |\hat{a}(t,\zz,\m,p_{v})|\lesssim|p_{v}|,\qquad
  |\hat{a}(t,\zz,\m,p_{v})-\hat{a}(t,\zz',\m',p_{v}')|\lesssim |\zz-\zz'|+\mathcal{W}_{1}(\m,\m')+|p_{v}-p_{v}'|.
\end{align}
\end{assumption}
Notice that the unique minimizer of the Hamiltonian 
depends only on the component $p_v$, rather than on the full adjoint vector $p$. Therefore, in
this framework, non-degeneracy is required solely for the diffusion coefficient of $V$.

Given the minimizer $\hat{a}$, we consider the FBSDE associated to
\eqref{frozen_SOC} 
\begin{equation}\label{linFBSDE}
  \begin{cases}
  dV_{t}=b\left(t,\VX_{t},\mu_t,\hat{a}(t,\VX_t,\mu_t,\sigma^{-1}(t,\VX_t,\mu_t)\BZ_t)\right)dt+\sigma(t,\VX_{t},\mu_t)dW_{t},\\
  dX_{t}=V_{t}dt,\\
  dY_t = - f\left(t,\VX_t,\mu_t,\hat{a}(t,\VX_t,\mu_t,\sigma^{-1}(t,\VX_t,\mu_t)\BZ_t)\right)dt + \BZ_t dW_{t},\\
  \VX_{0}= \eta,\quad Y_T = g(\VX_T,\mu_T),
\end{cases}
\end{equation}
and the corresponding HJB equation
\begin{equation}\label{HJB}
 \begin{cases}
  \KK u^{\m}(t,\zz)+H\left(t,\zz,\mu_t,\partial_v u^{\m}(t,\zz),\hat{a}(t,\zz,\mu_t,\partial_v u^{\m}(t,\zz))\right)=0,\qquad t\in(0,T),\ \zz=(v,x)\in\R^{2},\\
  u^{\m}(T,\zz) = g(\zz,\mu_T),
 \end{cases}
\end{equation}
where $\KK$ is the ultraparabolic operator in $\R^{3}$
\begin{equation}\label{oper1}
  \KK:=\frac{\sigma^2(t,\zz,\mu_t)}{2}\partial_{vv} + v\partial_x+\partial_t.
\end{equation}
Problems \eqref{linFBSDE} and \eqref{HJB} are tightly connected. If the HJB problem
admits a classical solution $u^{\m}$, 
then $u^{\m}$ acts as the {\it decoupling field} for the FBSDE \eqref{linFBSDE}: this means (cf.
Definition \ref{decf}) that the solution $(\VX,Y,\BZ)$ of the FBSDE \eqref{linFBSDE} satisfies
  $$Y_t =u^{\m}(t,\VX_t), \qquad \BZ_t = \sigma(t,\VX_t,\mu_t) \partial_v u^{\m}(t,\VX_t).$$
It then follows that the optimal state process $\VX$ is also characterized as the solution of the
forward SDE
\begin{equation}\label{linFSDE1}
 \begin{cases}
  dV_t =b(t,\VX_{t},\mu_t,\hat{a}(t,\VX_t,\mu_t,\p_v u^{\m}(t,\VX_t)))dt+
  \sigma(t,\VX_{t},\mu_t)dW_{t},\\
  dX_t = V_t dt.
 \end{cases}
\end{equation}
In the next subsection, we present a well-posedness result and introduce the functional framework
required for the study of a class of semilinear ultraparabolic problems, encompassing the HJB
equation \eqref{HJB} as a particular case.

\subsection{Semilinear kinetic equations}
We consider the Cauchy problem:
\begin{equation}\label{Cauchy}
 \begin{cases}
  \K u(t,v,x)=\ff\left(t,v,x,u(t,v,x),\partial_v u(t,v,x)\right), &(t,v,x)\in(0,T)\times\R^{2},\\
  u(T,\cdot)=\gg, &\text{on }\R^2.
\end{cases}
\end{equation}
Here, $\K$ denotes the {\it Langevin operator}, defined by
\begin{equation}\label{oper}
  \K:=\mathbf{a}(t,v,x)\partial_{vv} + v\partial_x+\partial_t.
\end{equation}
The HJB problem \eqref{HJB} fits into the framework of \eqref{Cauchy} by fixing a flow $\m$ and
setting
  $$\mathbf{a}(t,v,x)=\frac{1}{2}\s^{2}(t,v,x,\m_{t}),\quad \ff(t,v,x,u,p)=
  -H\left(t,v,x,\mu_t,p,\hat{a}(t,v,x,\mu_t,p)\right),\quad \gg(v,x)=g(v,x,\m_{T}).$$
To the best of our knowledge, the degenerate Cauchy problem \eqref{Cauchy} in this level of
generality has not been treated in the literature. Related, though distinct, contributions that
exploit the hypoelliptic structure of the operator include \cite{MR1941093,MR2048401}, and more
recently \cite{MR4717764,anceschi2024}.

As in the classical parabolic theory, the study of~\eqref{Cauchy} relies on a priori estimates in
appropriate functional spaces. Over the last few decades, significant advances have been made in
the regularity theory for kinetic operators, and distinct notions of H\"older regularity have been
proposed in the literature. Two main streams of research can be distinguished: one is based on the
notion of \emph{intrinsic} H\"older continuity, first introduced in~\cite{MR1386366,Manfredini},
whereas the other relies on \emph{anisotropic} H\"older spaces, defined in~\cite{Lunardi1997}
(see~\cite{MR4660246} for a comparison between these two frameworks). Several sharp Schauder
estimates have recently been established using these definitions (see,
e.g.,~\cite{MR4676647,MR4924753,MR4696151}). The present analysis relies on the results
in~\cite{MR4660246,MR3429628,Luce2023} which provide \emph{optimal} Schauder estimates, ensuring
the strongest H\"older regularity under the weakest assumptions on the coefficients. This choice
is further motivated in Remark~\ref{r1}.

To streamline the presentation, we state our main result first and postpone the detailed
definition of the anisotropic and intrinsic H\"older spaces, denoted by $C^{\a}_{\K,T}$ and
$\C^{\a}_{\K,T}$ respectively, to the end of this section.

\begin{assumption}\label{ass7}\hfill
\begin{enumerate}
\item The diffusion coefficient is uniformly elliptic and anisotropically H\"older continuous
  in $z=(v,x)$, uniformly with respect to $t$, i.e. there exists a positive constant $\k$ such that
  $\k^{-1}\le \mathbf{a}(t,z)\le \k$ for all $(t,z)\in(0,T)\times\R^{2}$, and
  $\mathbf{a}\in C^{\bar{\a}}_{\K,T}$ for some $\bar{\a}\in(0,1]$.
  \item The coefficients satisfy the growth and Lipschitz conditions
\begin{align}
   |\ff(t,z,u,p)|
   &\lesssim 1+|(z,u,p)|,\qquad
   |\gg(z)|\lesssim 1,\\
   |\ff(t,z,u,p)-\ff(t,z',u',p')|
   &\lesssim \left(1+|p|\right)\left(|z-z'|+|u-u'|+|p-p'|\right),\\
   |\gg(z)-\gg(z')|&\lesssim|z-z'|,
\end{align}
for any $z,z'\in\R^{2}$, $u,u',p,p'\in\R$ and $t\in(0,T)$.
\end{enumerate}
\end{assumption}

\begin{theorem}\label{t1}
Under Assumption \ref{ass7}, the Cauchy problem \eqref{Cauchy} admits a unique classical solution
$u\in \C^{2+\a}_{\K,t}$
for any $\a\in(0,\bar{\a})$ and $t\in[0,T)$, 
such that
\begin{align}\label{est1}
  |u(t,z)-u(t,z')|\lesssim |z-z'|,\qquad
  |u(t,z)-u(t',z)|\lesssim \sqrt{|t-t'|}\left(1+|z|\right),
\end{align}
for any $z,z'\in\R^{2}$ and $t,t'\in(0,T)$.
\end{theorem}
Theorem \ref{t1} may be proved by a suitable adaptation of the methods developed in
\cite{MR2048401}; we defer the complete proof to forthcoming work \cite{pas_ron}.

We proceed by recalling the definitions relevant to our setting. In \cite{LanconelliPolidoro1994}
it was shown that, when the diffusion coefficient $\mathbf{a}$ is constant, $\K$ is invariant with
respect to the left translations in the Lie group $(\R^{3},\circ)$ with composition law defined by
  $$(t',v',x')\circ(t,v,x)=(t+t',v+v',x+x'+tv');$$
moreover, $\K$ is homogeneous of degree two with respect to the dilations group
 $$\d_{\l}(t,v,x)=(\l^{2}t,\l v,\l^{3}x).$$
It is then natural to consider the {\it anisotropic norm}
 $$|(v,x)|_{\K}=|v|+|x|^{\frac{1}{3}},\qquad (v,x)\in\R^{2}.$$
\begin{definition}[\bf Anisotropic H\"older spaces]\label{anisotropic}
We set
  $$C^{\a}_{\K,T}:=L^{\infty}((0,T);C^{\a}_{\K}),\qquad \a\in(0,3],$$
where $C^{\a}_{\K}$ is the anisotropic H\"older space consisting of all measurable functions
$\gg:\R^{2}\to\R$ for which the following norm is finite:
\begin{align}
 \|\gg\|_{C^{\a}_{\K}}:=
 \begin{cases}
  \|\gg\|_{L^{\infty}(\R^{2})}+
  \sup\limits_{(v,x),(v',x')\in\R^{2}}\frac{|\gg(v,x) - \gg(v',x')|}{|(v-v',x-x')|_{\K}^{\a}},
  &\a\in(0,1],\\
 \|\gg\|_{L^{\infty}(\R^{2})}+\|\p_{v}\gg\|_{C^{\a-1}_{\K}}+
 \sup\limits_{v,x,x'\in\R}
 \frac{|\gg(v,x) - \gg(v,x')|}{|x-x'|^{\frac{\a}{3}}},
 &\a\in(1,3].
 \end{cases}
 \end{align}
\end{definition}
A crucial component of the definition of {\it intrinsic} regularity is the regularity along the
drift vector field
\begin{equation}\label{Y}
 \Y:=v\p_{x}+\p_{t}
\end{equation}
of $\K$: for any function $\ff$ defined on $(0,T)\times\R^{2}$ we set
  $$\|\ff\|_{\C^{\a}_{\Y,T}}:=\sup_{(t,v,x)\in (0,T)\times\R^{2}}\sup_{s\in(t,T)}\frac{|\ff(s,v,x+(s-t)v)-
  \ff(t,v,x)|}{|s-t|^{\frac{\a}{2}}},
  \qquad \a\in(0,2].$$
Following \cite{MR4660246}, we introduce the intrinsic H\"older spaces.
\begin{definition}[\bf Intrinsic H\"older spaces]
For $\a\in(0,3]$, the intrinsic H\"older space $\C^{\a}_{\K,T}$ is defined by the norm
\begin{align}
  \|\ff\|_{\C^{\a}_{\K,T}}:=\begin{cases}
  \|\ff\|_{C^{\a}_{\K,T}}
  +\|\ff\|_{\C^{\a}_{\Y,T}},&\hspace{0pt}\quad \a\in(0,1],\\
 \|\ff\|_{C^{\a}_{\K,T}}+\|\p_{v}\ff\|_{C^{\a-1}_{\K,T}}+\|\ff\|_{\C^{\a}_{\Y,T}},&\hspace{0pt}\quad \a\in(1,2],\\
  \|\ff\|_{C^{\a}_{\K,T}}+\|\p_{v}\ff\|_{C^{\a-1}_{\K,T}}+\| {\Y\ff}
 \|_{C^{\a-1}_{\K,T}}, &\hspace{0pt}\quad \a\in(2,3].
  \end{cases}
\end{align}
Higher order intrinsic H\"older spaces are defined recursively.
\end{definition}
\begin{remark}\label{r1}
Intrinsic regularity is stronger than anisotropic regularity; in particular, we have the strict
inclusion $$\C^{\a}_{\K,T}\subsetneq C^{\a}_{\K,T}.$$ Moreover, if $u\in \C^{2}_{\K,T}$, then the
Euclidean derivatives $\partial_v u$, $\partial_{vv}u$ and the directional derivative $\Y u$ exist
in the classical sense. Therefore, Theorem~\ref{t1} provides a classical solution to the Cauchy
problem~\eqref{Cauchy} under the sole assumption that the coefficients are anisotropically regular
(in particular, they may be merely measurable in time).

Even more importantly, for $u\in \C^{2}_{\K,T}$ the It\^o formula holds, which is the fundamental
tool in the analysis of the associated stochastic optimal control problem. By contrast, if $u\in
C^{2}_{\K,T}$, then $\K u$ can, in general, be interpreted only in a weak sense. We notice
explicitly that the existence of the Euclidean derivatives $\partial_x u$ and $\partial_t u$ is
ensured only under the stronger regularity assumption $u\in \C^{3}_{\K,T}$.
\end{remark}

\subsection{McKean-Vlasov optimal control problem}
In this section we introduce the standing assumptions on the coefficients and state our main
result on the well-posedness of the kinetic McKean-Vlasov optimal control problem \eqref{MKVCP}.

\begin{assumption}\label{ass1}
There exists positive constants $\k$ and $\bar{\a}\in(0,1]$, such that the following conditions
hold for any $t\in[0,T]$, $\zz,\zz'\in\times\R^{2}$, $\m\in\mathcal{P}_2(\R^{2})$ and $a,a'\in A$.
\begin{enumerate}
\item \emph{Non-degeneracy and growth conditions:}
\begin{align}
  &\k^{-1}\le \s(t,\zz,\m)\le\k,\qquad &|g(\zz,\m)|\lesssim 1,\\
  &|b(t,\zz,\m,a)|\lesssim 1+|a|,\qquad &|\p_a b(t,\zz,\m,a)|\lesssim 1,\\
  &|f(t,\zz,\m,a)|\lesssim 1+|a|^{2},\qquad &|\p_{a} f(t,\zz,\m,a)|\lesssim 1+|a|.
\end{align}
\item \emph{Regularity:} the diffusion coefficient $\s$ is differentiable in $v$ and we have
  $$|\p_{v} \s(t,\zz)-\p_{v} \s(t,\zz')|\lesssim |\zz-\zz'|^{\bar{\a}}_{\K},\qquad |\s(t,\zz,\m)-\s(t,\zz',\m)|+|g(\zz,\m)-g(\zz',\m)|\lesssim |\zz-\zz'|.$$
Moreover, $b,f$ are differentiable in $v,a$ and satisfy
\begin{equation}\label{and4}
\begin{split}
  |b(t,\zz,\m,a)-b(t,\zz',\m,a)|+|f(t,\zz,\m,a)-f(t,\zz',\m,a)|&\lesssim |\zz-\zz'|,\\
  |\p_{v} b(t,\zz,\m,a)-\p_{v} b(t,\zz',\m,a')|+|\p_{v} f(t,\zz,\m,a)-\p_{v} f(t,\zz',\m,a')|&\lesssim |\zz-\zz'|^{\bar{\a}}_{\K}+|a-a'|^{\bar{\a}},\\
  |\p_a b(t,\zz,\m,a)-\p_a b(t,\zz',\m,a')|&\lesssim |\zz-\zz'|^{\bar{\a}}_{\K}+|a-a'|^{\bar{\a}},\\
  |\p_a f(t,\zz,\m,a)-\p_a f(t,\zz',\m,a')|&\lesssim |\zz-\zz'|^{\bar{\a}}_{\K}+(1+|a|)|a-a'|^{\bar{\a}}.
\end{split}
\end{equation}
{The minimizer $\hat{a}=\hat{a}(t,\zz,\m,p_{v})$ is differentiable in $v,p_{v}$ and satisfies}
\begin{align}
  |\p_{v}\hat{a}(t,\zz,\m,p_{v})-\p_{v}\hat{a}(t,\zz',\m,p_{v}')|&\lesssim
  |\zz-\zz'|^{\bar{\a}}_{\K}+|p_{v}-p_{v}'|,\\
  |\p_{p_{v}}\hat{a}(t,\zz,\m,p_{v})-\p_{p_{v}}\hat{a}(t,\zz',\m,p_{v}')|&\lesssim |\zz-\zz'|^{\bar{\a}}_{\K}+
  |p_{v}-p_{v}'|^{\bar{\a}}.
\end{align}
  \item \emph{Lipschitz continuity in the measure argument:}
\begin{equation}\label{w1}
\begin{split}
  |\s(t,\zz,\m)-\s(t,\zz,\m')|+|g(\zz,\m)-g(\zz,\m')|&\lesssim\mathcal{W}_{1}(\m,\m'),\\
  |f(t,\zz,\m,a)-f(t,\zz,\m',a)|+|b(t,\zz,\m,a)-b(t,\zz,\m',a)|&\lesssim\mathcal{W}_{1}(\m,\m').
\end{split}
\end{equation}
\end{enumerate}
\end{assumption}
The main theoretical result of the paper is the following
\begin{theorem}[\bf Well-posedness for the McKean-Vlasov optimal control problem \eqref{MKVCP}]\label{th3}
Under Assumptions \ref{ass2} and \ref{ass1}, for any $\eta\in L^2(\O,\mathcal{F}_0, P)$ there
exists a unique strong solution $(\VX^{\y},Y^{\y},\BZ^{\y})$ to the McKean-Vlasov FBSDE
\begin{equation}\label{MKVFBSDE}
    \begin{cases}
  dV_{t}=b\left(t,\VX_{t},[\VX_{t}],\hat{a}(t,\VX_t,[\VX_{t}],\sigma^{-1}(t,\VX_t,[\VX_{t}])\BZ_t)\right)dt+\sigma(t,\VX_{t},[\VX_{t}])dW_{t},\\
  dX_{t}=V_{t}dt,\\
  dY_t = - f\left(t,\VX_t,[\VX_{t}],\hat{a}(t,\VX_t,[\VX_{t}],\sigma^{-1}(t,\VX_t,[\VX_{t}])\BZ_t)\right)dt + \BZ_t dW_{t},\\
  \VX_{0}= \eta,\quad Y_T = g(\VX_T,[\VX_{T}]).
\end{cases}
\end{equation}
Moreover:
\begin{itemize}
  \item there exists a continuous function $u\in\Lip_{0,T}(\R^{2})\cap\C^{3+\a}_{\K,t}$
  for any $\a\in(0,\bar{\a})$ and $t\in[0,T)$, such that the processes $\left(u(t,\VX^{\y}_t)\right)_{t\in[0,T]}$ and $Y^{\y}$ are indistin\-guishable;
  \item the pair $(\VX^{\y},\hat{a})$, with $\hat{a}_t :=\hat{a}(t,\VX^{\y}_t,[\VX^{\y}_t],\s^{-1}(t,\VX^{\y}_t,[\VX^{\y}_t])\BZ^{\y}_t)$,
  is the solution of the McKean-Vlasov optimal control problem \eqref{MKVCP}. In particular,
    $$J(\hat{a})=E[u(0,\eta)].$$
\end{itemize}
\end{theorem}
\begin{remark}\label{r10}
Even if we restrict our attention to square-integrable solutions, in \eqref{w11}-\eqref{w1} we
require Lipschitz continuity with respect to the weaker norm $\mathcal{W}_{1}$, rather than
$\mathcal{W}_{2}$. This choice allows us to exploit the well-known Kantorovich-Rubinstein duality
(cf. \eqref{KR}), which provides linear dual estimates well suited for stability arguments. This
setting is also motivated by the use of a standard contraction argument to obtain a {\it unique}
fixed point, as opposed to more general Schauder fixed point results (see, e.g.,
\cite{MR3752669}), which typically yield existence but do not guarantee uniqueness.
\end{remark}

\section{Numerical tests}\label{sec:numerical-tests}
In this section we describe the stochastic dynamics used in the numerical experiments. We consider
a stylized yet fully consistent model, comprising two exogenous processes calibrated to historical
data for the electricity price and the net household load, coupled with controlled dynamics for
the battery power and the state of charge.
 The numerical experiments provide encouraging
evidence in support of the proposed approach, as the performance under the optimal control policy
compares favorably with that of the uncontrolled benchmark.

\subsection{Electricity price}\label{subsec:electricity_price}
The electricity spot price $S_t$ is modeled as a positive mean-reverting process.
Specifically, the log-price follows an Ornstein-Uhlenbeck (OU) dynamics:
\begin{equation}\label{eq:log-price-ou}
d\log S_t = \kappa^S(t)\bigl(\mu^S(t) - \log S_t\bigr)dt + \sigma^S(t)dW^S_t,
\end{equation}
where $\kappa^S(t)>0$ is a piecewise-constant day/night mean-reversion speed, $\mu^S(t)$ is the
mean-reversion level, and $\sigma^S(t)$ is a piecewise-constant day/night volatility profile. The
explicit solution for constant $\kappa^{S}$ is
\begin{equation}\label{eq:price-solution}
S_t=\exp\!\left( e^{-\kappa^S t}\log S_0 +\int_0^t \kappa^S e^{-\kappa^S(t-u)}\hat{\mu}^S(u)du
+\int_0^t \sigma^S(u)e^{-\kappa^S(t-u)}dW^S_u \right),
\end{equation}
where $\hat{\mu}^S(t):=\mu^S(t)-\frac{(\sigma^S(t))^2}{2\kappa^S}$. Hence, $S_t$ is log-normal
with
\begin{equation}
E[\log S_t]
=
e^{-\kappa^S t}\log S_0 + \int_0^t \kappa^S e^{-\kappa^S(t-u)}\hat{\mu}^S(u)du,\qquad \var(\log
S_t)
=
\int_0^t (\sigma^S(u))^2 e^{-2\kappa^S(t-u)}du.
\end{equation}
The deterministic function $\mu^S(t)$ captures predictable seasonality at multiple time scales
(intraday peaks and weekly cycles) and represents the time-dependent level to which the log-price
reverts. In the numerical implementation, we approximate $\mu^S(\cdot)$ by a degree-$2$
trigonometric polynomial, while $\sigma^S(\cdot)$ and $\kappa^S(\cdot)$ are taken piecewise
constant with two regimes, corresponding to daytime (hours 8-19) and nighttime (all remaining
hours).

The coefficients are calibrated on historical time series of the Italian day-ahead PUN
price\footnote{Daily PUN values and interactive charts are publicly available on the GME web
page:\\ \url{https://www.mercatoelettrico.org/it-it/Home/Esiti/Elettricita/MGP/Esiti/PUN}\\
Source: Gestore dei Mercati Energetici S.p.A.} which represents the reference wholesale price for
electricity in Italy. Figure \ref{fig:pun_price} displays the observed PUN price (red) together
with the 80\% confidence bands generated by the calibrated model for two representative weeks in
June and December 2025.
\begin{figure}[h!]
    \centering
    \includegraphics[width=0.9\textwidth,height=0.2\textheight]{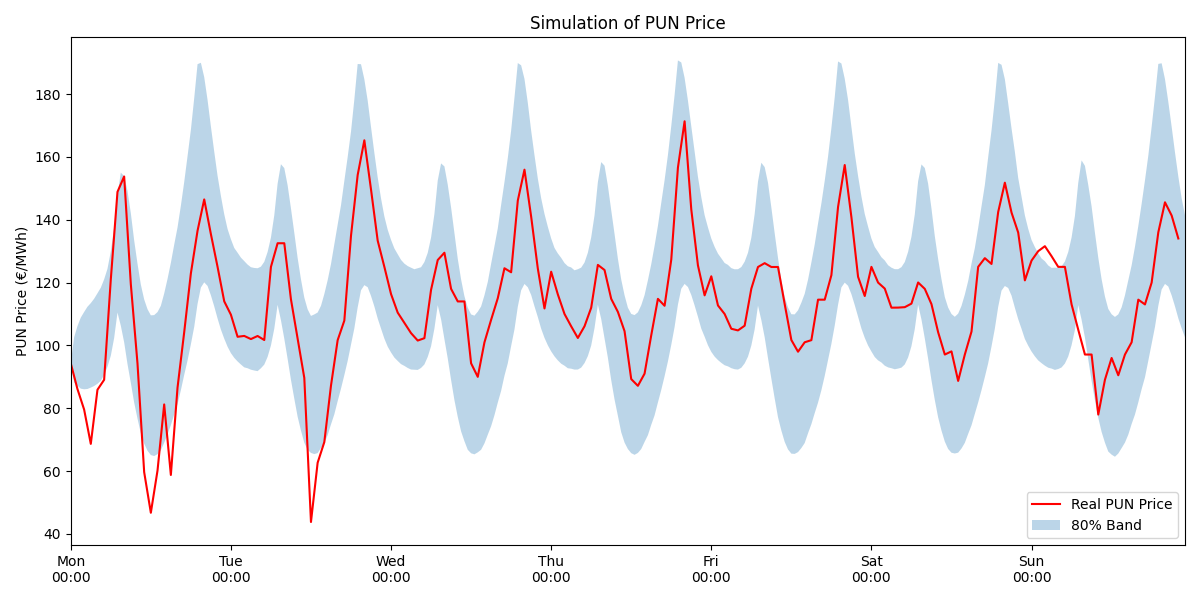}

    \vspace{-0.1cm}

    \includegraphics[width=0.9\textwidth,height=0.2\textheight]{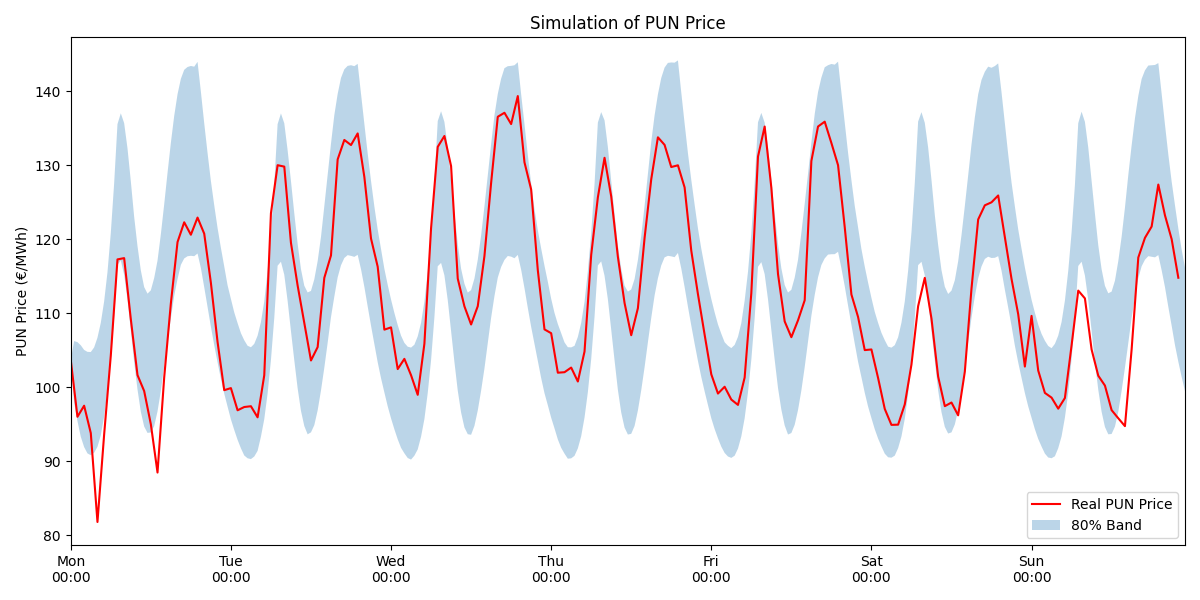}

    \caption{Second week of June 2025 and first week of December 2025: observed PUN price (red) and 80\% confidence bands of the calibrated model (light blue).}
    \label{fig:pun_price}
\end{figure}

\subsection{Net home load}\label{subsec:load}
The net home load $H_t$ is defined as household consumption minus local renewable generation.
Positive values correspond to net withdrawal from the grid, while negative values represent
surplus local production. The net home load is exogenous and modeled as a mean-reverting OU
process
\begin{equation}\label{eq:load-ou}
dH_t=\kappa^H(t)\bigl(\mu^H(t)-H_t\bigr)dt+\sigma^H(t)dW^H_t,
\end{equation}
where $\mu^H(t)$ is a trigonometric polynomial of degree $2$ capturing predictable seasonality,
while $\sigma^H(t),\kappa^H(t)>0$ follow a day/night specification.

We calibrate the model using a historical time series from a residential energy system equipped
with photovoltaic generation and battery storage\footnote{The test data are sourced from the
publicly available SolarEdge Monitoring platform:\\
{\url{https://monitoringpublic.solaredge.com/solaredge-web/p/home}}\\ and refer to a photovoltaic
installation located in Bologna, Italy, with a peak power of 5.16 kWp and a SolarEdge battery
storage system with a capacity of 9.7 kWh.}. Figure \ref{fig:inputs-combined} shows the dataset of
photovoltaic production in green and the household consumption in gray, with the resulting net
home load in red.
\begin{figure}[h!]
    \centering
    \includegraphics[width=0.8\textwidth]{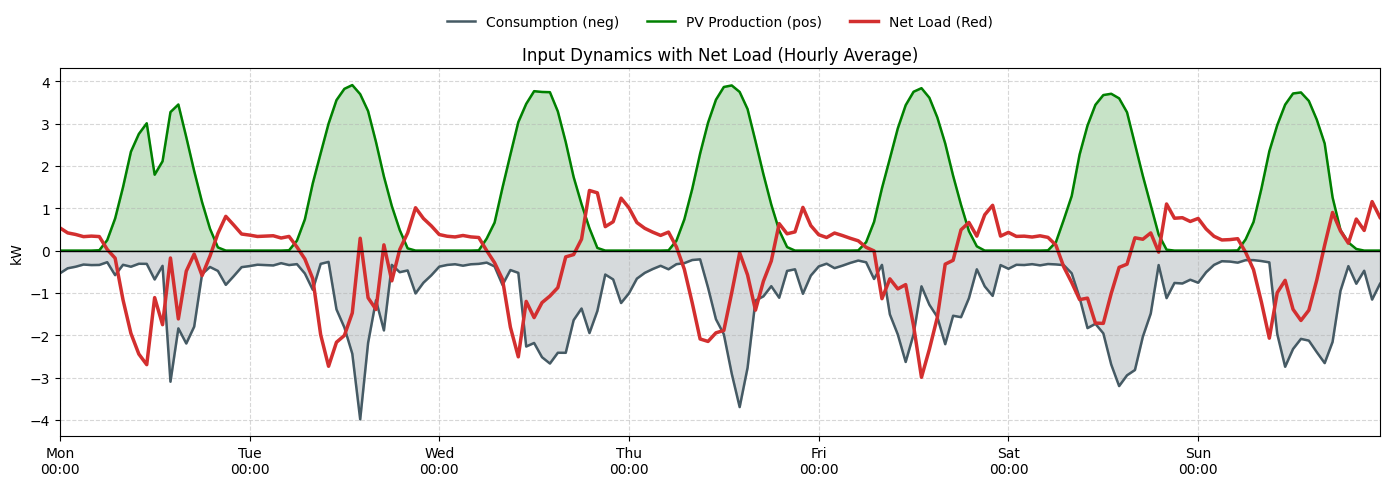}

    \vspace{-0.1cm}

    \includegraphics[width=0.8\textwidth]{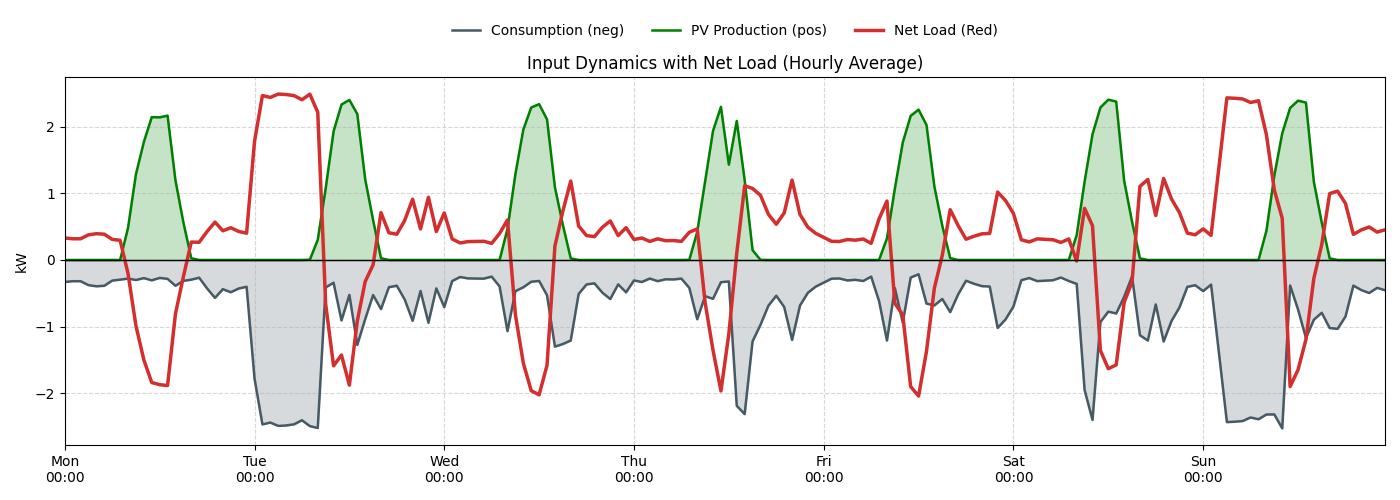}

    \caption{Representative weekly input profiles in 2025 (second week of June and first week of December): household consumption (gray), photovoltaic production (green), and net home load (red).}
    \label{fig:inputs-combined}
\end{figure}
In June, the net home load is visibly more stable, reflecting the relatively high and predictable
photovoltaic generation, whereas in December the production is more volatile, leading to a higher
variability of the net load.

The coefficients of the SDE \eqref{eq:load-ou} are calibrated on these datasets via
regression-based estimation. Figure \ref{fig:net_load} displays the observed net home load (red)
together with the $80\%$ confidence band of the calibrated model. We observe that, in June, net
load is largely driven by photovoltaic generation, whereas in December the observations exhibit
substantially higher variability.
\begin{figure}[h!]
    \centering
    \includegraphics[width=0.9\textwidth,height=0.2\textheight,keepaspectratio=false]{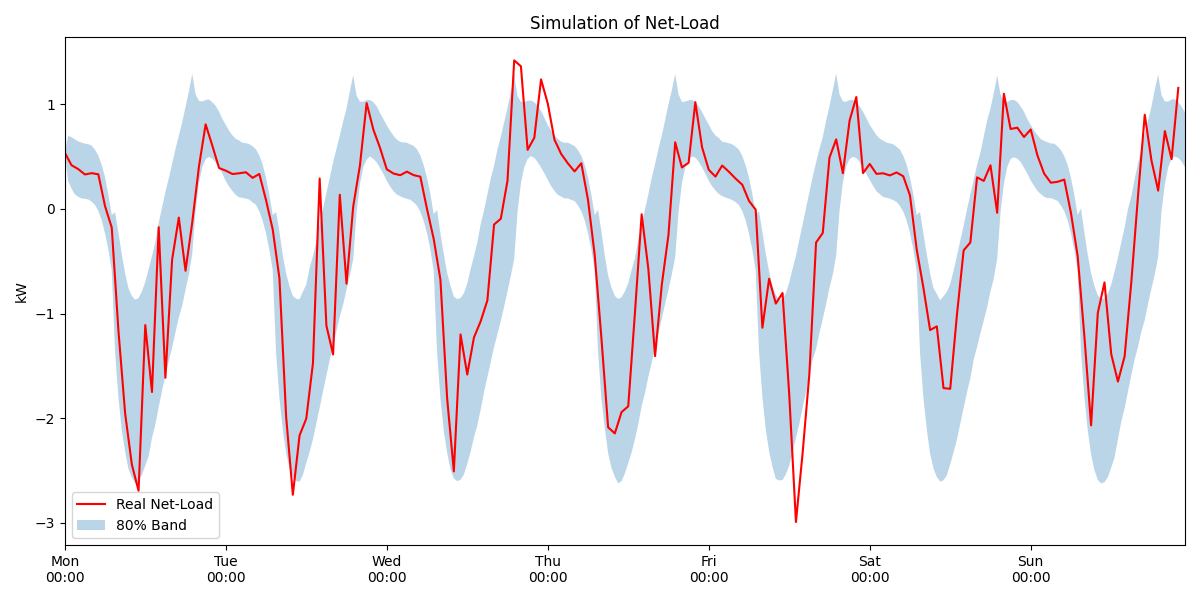}

    \vspace{-0.1cm}

    \includegraphics[width=0.9\textwidth,height=0.2\textheight,keepaspectratio=false]{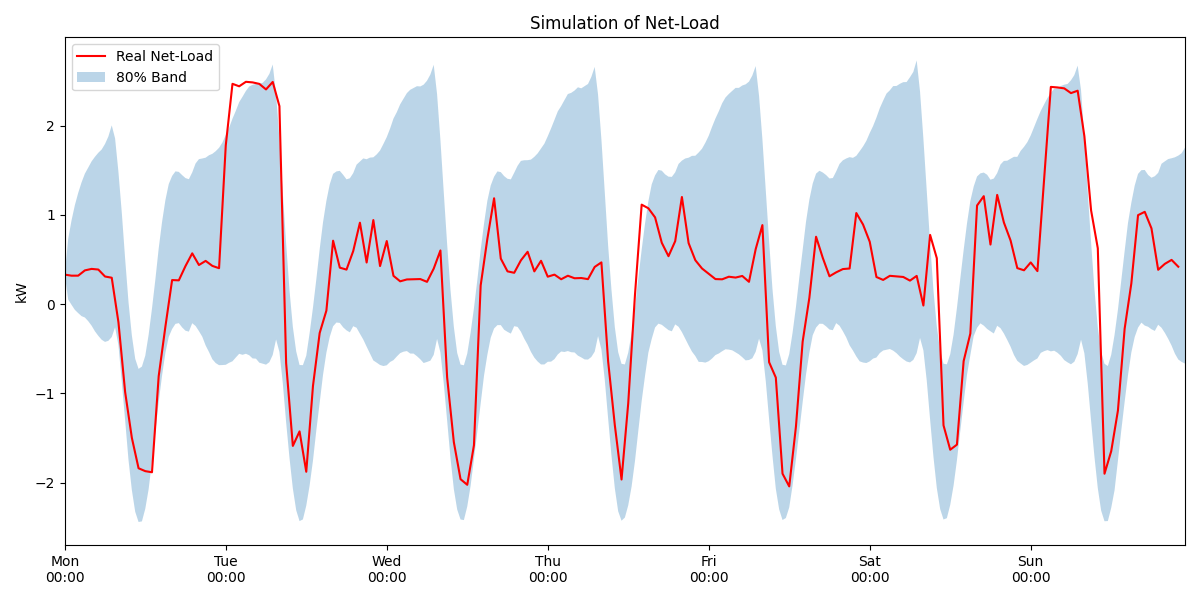}

    \caption{Second week of June 2025 and first week of December 2025: observed net load (red) and $80\%$ confidence bands of the calibrated model (light blue).}
    \label{fig:net_load}
\end{figure}

\subsection{Battery power and cost functional}
We adopt the following dynamics for the battery power $V_t$, with the convention that $V_t>0$
corresponds to charging and $V_t<0$ to discharging:
\begin{equation}\label{eq:V-controlled}
dV_t^a=b^V(\VX_t^a,a_t)dt+\sigma^V(\VX_t^a,[\VX_t^a]) dW^V_t, \qquad
\VX_t^a=(S_t,H_t,V_t^a,X_t^a).
\end{equation}
Here
\begin{align}
 b^V(\VX_t^a,a_t)&=
 a_t\phi(X_t^a)+c_1\psi(-X_t^a)\psi(V_{max}-V_t^a)-c_2\psi(X_t^a-X_{max})\psi(V_t^a-V_{min}),\\
 \sigma^V(\VX_t^a,[\VX_t^a]) &= \sigma_V +
\sigma_V^H|H_t|+\sigma_V^V|V^a_t|+\sigma_V^\kappa\left|V^a_t-E\left[V^a_t\right]\right|
\end{align}
and the saturation functions \begin{align*} \phi(x)=
\max\!\left\{0,\min\!\left\{1,\frac{x}{\delta},\frac{X_{\max}-x}{\delta}\right\}\right\},\qquad
\psi(x) = \max\left\{0,\min\left\{1,\frac{x}{\d}\right\}\right\},
\end{align*}
attenuate the control action as the state approaches the admissible SOC and power limits,
thereby mitigating boundary violations without enforcing hard state constraints.

We allow $\sigma^V$ to depend on $(S,H,V,X)$ in order to capture state-dependent uncertainty, for
instance random tracking or forecast errors. Specifically, we let volatility increase with
operating stress, i.e., under large net home load and at high charging/discharging power levels.
We could also incorporate asymmetric charging and discharging efficiencies, leading to dynamics of
the form $$dX_t^a=\left(\eta_c V_t^{a,+}-\eta_d^{-1}V_t^{a,-}\right)dt,$$ with
$\eta_c,\eta_d\in(0,1)$. Although economically realistic, this refinement leaves the theoretical
structure unchanged and is therefore omitted for simplicity.

\medskip
For the cost functional \eqref{e2}, we consider the running cost
\begin{equation}\label{e:def_f_num}
f(t,\VX_t^a,[\VX_t^a],a_t)
=
S_t(H_t+V_t^a) + \lambda_V |V_t^a|^2 + \lambda_a a_t^2 + h_{\textbf{bat}}(X_t^a-E[X_t^a]) +
h_{\textbf{con}}(H_t+V_t^a-E[H_t+V_t^a]),
\end{equation}
with $h_{i}(x)=c_{i}((x^{-})^2+2(x^+)^2)$ for some positive constants $c_{i}$, for
$i=\textbf{bat},\textbf{con}$. The term $S_tV_t^a$ drives the arbitrage mechanism, encouraging
charging at low prices and discharging at high prices. Quadratic penalties model battery
degradation and ramp-rate constraints. The mean-field terms penalize deviations from the
population average, promoting aggregate stabilization.

The terminal cost is
\begin{equation}\label{termc}
g(\VX_T^a,[\VX_T^a])
=
-\gamma S_T X_T^a + \frac{\omega}{2}\left(X_T^a-E\left[X_T^a\right]\right)^2.
\end{equation}
An alternative formulation would replace the quadratic mean-field penalty with a Wasserstein-$2$
distance between the empirical SOC distribution and a prescribed target law. This choice would
yield a genuinely distributional control objective, which is conceptually appealing and may be of
independent interest, but we leave it for future work.

\subsection{FBSNN and the optimal control}\label{sec:num-fbsde}
The optimization problem is numerically challenging due to the fully coupled and nonlinear
structure of the associated FBSDE combined with a mean-field type interaction. In particular,
since the running cost depends quadratically on the control, choosing the optimal strategy leads
to a backward component with quadratic growth in $\zeta$ and a forward dynamics that depends
explicitly on the backward variables. This leads to a coupled system associated with a
five-dimensional HJB equation.

More precisely, the stochastic control problem is associated to the following mean-field FBSDE
\begin{equation}
\begin{cases}
  dV_{t}=b^{V}(\VX_{t},\hat{a}(\VX_{t},\zeta_t))dt+\s^V(\VX_{t},[\VX_{t}]) dW^{V}_{t},\qquad\qquad \VX_t=(S_t, H_t, V_t, X_t),\\
  dX_{t}=V_{t}dt. \\
  dY_t = -f(t,\VX_{t},[\VX_{t}], \hat{a}(\VX_{t},[\VX_{t}],\zeta_t))dt + \zeta_tdW_t, \\
  Y_T = g(\VX_{T},[\VX_{T}]),
\end{cases}
\label{eq:fbsde_system}
\end{equation}
where $\hat{a}(\VX_t,\zeta_t):=-\frac{\phi(X_t)\zeta^V_t}{2\l_a\s^V(\VX_{t},[\VX_{t}])}$ and $(S,H)$ are the exogenous processes defined in \eqref{eq:price-solution} and \eqref{eq:load-ou}.
To solve \eqref{eq:fbsde_system}, we adopt a slight modification of the deep learning approach
introduced in \cite{MR4478485}, inspired by \cite{Raissi_2018}. The key idea is to approximate the
backward component over the full time horizon by the neural representation $$Y_t \approx
u_\theta(t,\VX_t),$$ where $u_\theta$ is a feed-forward neural network with parameters $\theta$.

A key structural feature of our implementation is that the process $\zeta_t$ is not parameterized
by a separate network. Instead, it is recovered via automatic differentiation:
$$\zeta_t=\nabla_{(S,H,V)}u_\theta(t,\VX_t)\Sigma(t,\VX_t,[\VX_t]),$$ where $\Sigma$ is a
$3\times3$ diagonal matrix with $(\s^S,\s^H,\s^V)$ on the diagonal. This design enforces
consistency between $Y$ and $\zeta$, keeps the martingale term aligned with the value function
approximation, reduces the number of trainable parameters, and improves numerical stability in the
presence of quadratic growth in $\zeta$.

The dependence on the law $[\VX_t]$ is handled through a particle approximation. At each training
step, we simulate $N_{\text{sim}}$ trajectories and approximate the law of $\VX_t$ by its
empirical measure, $$ [\VX_t] \approx \frac{1}{N_{\text{sim}}} \sum_{i=1}^{N_{\text{sim}}}
\delta_{\VX_t^{(i)}}. $$ In particular since the dependence on the law is only through the mean we
may just calculate it as empirical average across the simulated batch. In this way, the mean-field
FBSDE is approximated by a large but finite interacting particle system. As $N_{\text{sim}}$
increases, the approximation converges in the sense of propagation of chaos.

We discretize the time interval $[0,T]$ on a uniform grid $t_n=n\Delta t$, $n=0,\dots,N$, with
$\Delta t=T/N$. Let $$Y_{t_n}=u_\theta(t_n,\VX_{t_n})$$ denote the network output evaluated at
$(t_n,\VX_{t_n})$. The forward and backward components are then approximated by a {\it forward}
Euler-Maruyama scheme:
\begin{align*}
\VX_{t_{n+1}} &= \VX_{t_n} + b_{t_n}\Delta t + \sigma_{t_n}\Delta W_n,\\ \tilde{Y}_{t_{n+1}} &=
Y_{t_n} - f_{t_n}\Delta t + \zeta_{t_n}\Delta W_n,
\end{align*}
where $\Delta W_n=W_{t_{n+1}}-W_{t_n}$ and $b_{t_n},\sigma_{t_n},f_{t_n},\zeta_{t_n}$ are
evaluated at $(t_n,\VX_{t_n})$ (and, when applicable, at the corresponding mean-field terms).

Training amounts to minimizing the mismatch between the network outputs and the one-step
discretized backward dynamics, together with the terminal condition: $$
\sum_{n=0}^{N-1}\left\|Y_{t_{n+1}}-\tilde{Y}_{t_{n+1}}\right\|^2
+\left\|Y_{t_N}-g(\VX_{t_N})\right\|^2 +\left\|\zeta_{t_N}-\nabla
g(\VX_{t_N})\sigma_{t_N}\right\|^2, $$ where $\|\cdot\|$ is the Euclidean norm on
$\R^{N_{\text{sim}}}$. The last term enforces consistency at terminal time also at the gradient
level, which typically improves numerical stability. Given the strong coupling and quadratic
growth of the driver, standard numerical stabilization techniques are employed, such as gradient
clipping, antithetic sampling and a multi-stage learning rate schedule.

The neural network $u_\theta:[0,T]\times\R^4\to\R$ is parameterized by $\theta$ and takes as input
the concatenation of time and state variables $(t,\VX_t)\in\R^5$. Its scalar output provides an
approximation of the value process $Y_t$. We use a fully connected architecture with an input
layer of dimension $1+4$ (time plus state), four hidden layers with 256 neurons each, and a
one-dimensional linear output layer. All hidden layers employ the sine activation function. Smooth
periodic activations are well suited to approximating solutions of semilinear parabolic PDEs,
which are typically smooth in both time and space. In particular, compared with piecewise-linear
activations, sine activations yield smoother gradients, which is advantageous when spatial
derivatives are computed via automatic differentiation; moreover, unlike sigmoid-type
nonlinearities, they are less prone to vanishing gradients. Network weights are initialized via
Xavier (Glorot) initialization, which scales the variance of the initial weights with the layer
width. This choice helps prevent early saturation of the sine activations and promotes stable
signal propagation across depth. Overall, the architecture is tailored to approximate a smooth
value function while preserving the structural coupling of the FBSDE through automatic
differentiation.

\subsection{Numerical experiments}\label{sec:numerical-experiments}
Table \ref{tab:parameters} reports the set of parameters adopted in our numerical experiments.
They are selected so that the grid exchange component carries a slightly larger weight than the
other cost terms, thereby making grid interactions the primary driver of the optimization while
leaving the remaining penalties as secondary regularization effects.
\begin{table}[h!] \centering
\begin{tabular}{|cc|cc|cc|cc|cc|}
\hline Parameter & Value & Parameter & Value & Parameter & Value & Parameter & Value & Parameter &
Value \\ \hline $c_1$ & 10 & $c_2$ & 10 & $\lambda_V$ & 0.001 & $\lambda_a$ & 0.01 & $\gamma$ & 1
\\ $\delta$ & 1 & $X_{\max}$ & 10 & $V_{\max}$ & 2 & $V_{\min}$ & -2 & $\sigma_V$ & 0.01 \\
$\omega$ & 0.01 & $c_{\textbf{bat}}$ & 0.0001 & $c_{\textbf{con}}$ & 0.01 & $T$ & 24 &
$N_{\text{sim}}$ & 10000 \\ $N$ & 150 & $\sigma_V^H$ & 0.001 & $\sigma_V^V$ & 0.001 &
$\sigma_V^\kappa$ & 0.001 &  &  \\ \hline
\end{tabular}
\caption{Model parameters used in the simulations.}
\label{tab:parameters}
\end{table}


Firstly, we solve the stochastic control problem numerically using the mean-field FBSDE
formulation described in Section~\ref{sec:num-fbsde}. Once training is completed, the neural
network yields an optimal feedback control $a_t^*$ and, in turn, generates controlled trajectories
$(\VX_t^*)$ that are optimal for the objective functional \eqref{e2}.

For comparison, we then consider a passive regime in which the battery is operated under a
deterministic feedback policy aimed solely at compensating the net home load. No strategic
optimization is performed. In this setting, the power $V_t$ is prescribed so as to enforce the SOC
constraints:
\begin{equation}
V_t=
\begin{cases}
-\min\{H_t,0\}, & X_t \le 0,\\ -\max\{H_t,0\}, & X_t \ge X_{\max},\\ -H_t, & \text{otherwise}.
\end{cases}
\end{equation}
Hence, whenever feasible, the battery offsets the net home load; charging or discharging is
inhibited at the SOC boundaries.

In Figure~\ref{fig:J_comparison} we compare the accumulated costs of the uncontrolled benchmark
and of the optimal controlled strategy. In the uncontrolled case, the accumulated cost is computed
by setting $\l_a=0$ in the running cost \eqref{e:def_f_num} and using the terminal contribution
\eqref{termc}. We observe a systematic reduction in the mean accumulated cost under optimization,
together with a tighter dispersion across trajectories, indicating that the learned policy
delivers
not only lower expected cost but also more stable performance.
\begin{figure}[h!]
    \centering
    \includegraphics[width=0.8\textwidth,height=0.2\textheight,keepaspectratio=false]{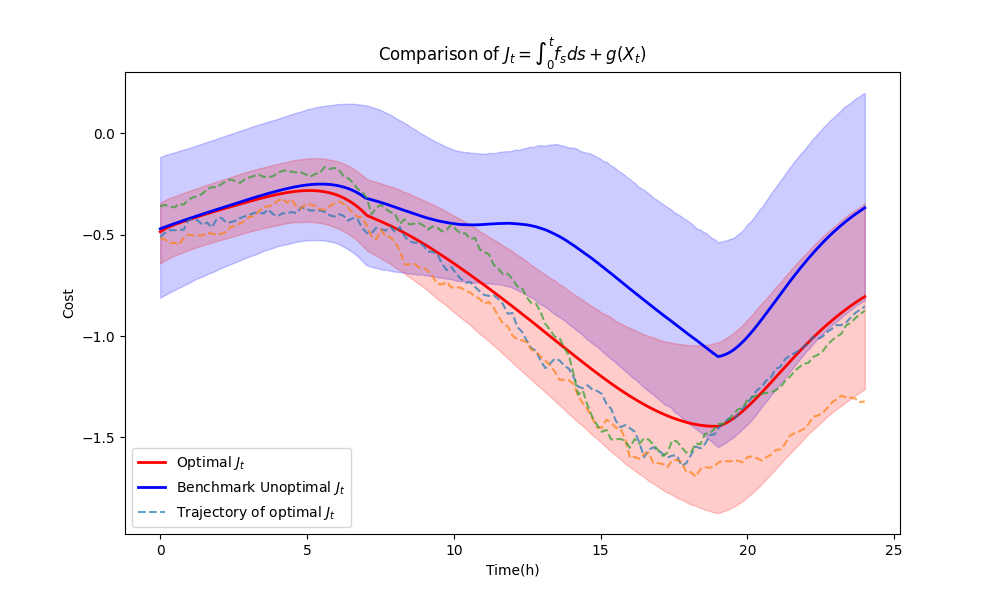}

    \vspace{-0.1cm}

    \includegraphics[width=0.8\textwidth,height=0.2\textheight,keepaspectratio=false]{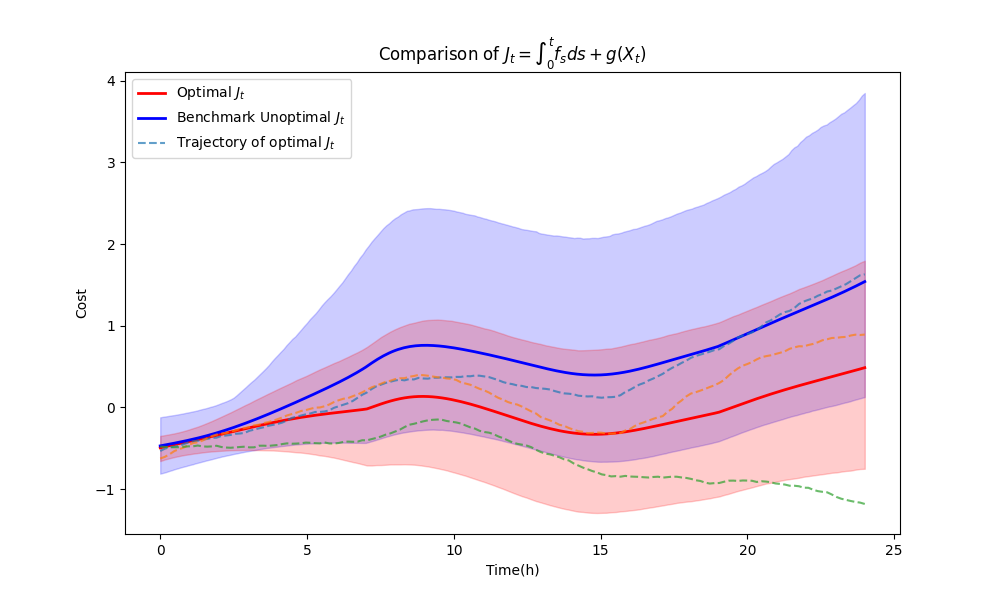}

\caption{Second week of June 2025 and first week of December 2025: comparison of the accumulated
cost. Monte Carlo mean for the uncontrolled benchmark (blue) and for the optimal control strategy
(red), with a $90\%$ percentile band (shaded) and three representative sample trajectories under
optimal control.}\label{fig:J_comparison}
\end{figure}

For a more refined analysis, Figure \ref{fig:forward_components} reports the time evolution of the
main state variables under the optimal control and the uncontrolled benchmark, for both the June
and December datasets. In both cases, the SOC remains well within its admissible range.
Remarkably, despite delivering a lower total cost, the controlled strategy systematically seeks to
minimize the use of the battery, thereby effectively mitigating degradation; this behaviour is
particularly evident in comparison with the uncontrolled benchmark. Moreover, the mean-field
penalties enhance grid stabilization, as reflected by the reduced amplitude of the bands. Finally,
the optimal charging power tracks price fluctuations, charging during low-price periods and
discharging during price peaks.

Overall, the mean-field FBSDE framework delivers both economic improvements and enhanced stability
relative to the naive benchmark, suggesting that the proposed methodology is a promising
direction. These results provide a strong motivation to further investigate the approach within
richer and more realistic models, including additional operational constraints and more detailed
system dynamics.
\begin{figure}[h!]
    \centering
    \begin{minipage}{0.48\textwidth}
        \centering
        \includegraphics[width=\textwidth]{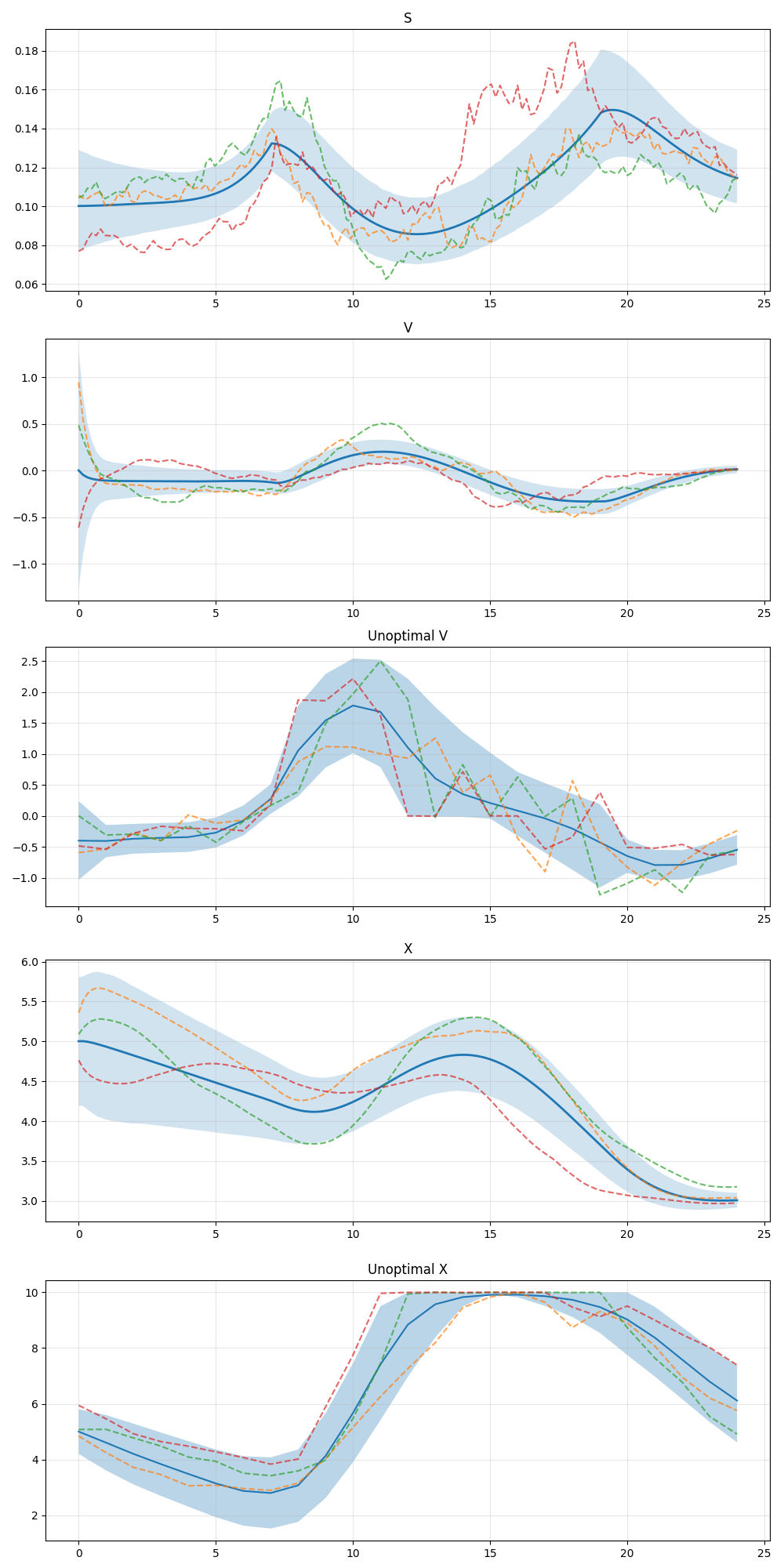}
        \caption*{(a) June 2025}
    \end{minipage}%
    \hfill
    \begin{minipage}{0.48\textwidth}
        \centering
        \includegraphics[width=\textwidth]{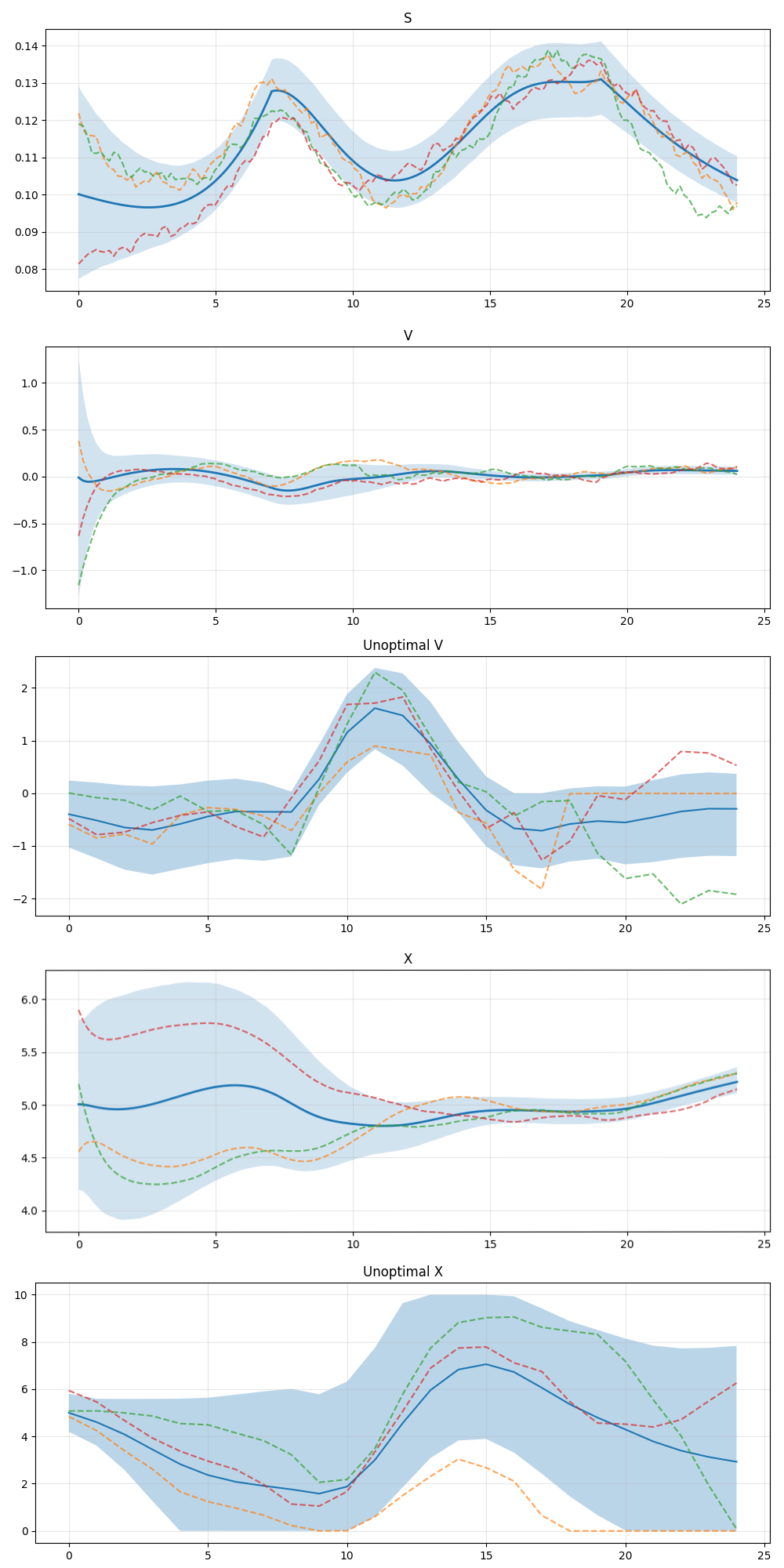}
        \caption*{(b) December 2025}
    \end{minipage}
    \caption{Second week of
    June 2025 and of December 2025: the $(S,V,X)$ forward components of
    the optimal trajectory and the $(V,X)$ forward components of the unoptimal trajectory; Monte Carlo mean (light blue), $90\%$-percentile band (shaded) and three sampled trajectories.}
    \label{fig:forward_components}
\end{figure}

\section{Proof of Theorem \ref{th3}}\label{sec4}
\subsection{Kinetic FBSDEs}
As a first step, we establish a general well-posedness result for kinetic FBSDEs of the
form
\begin{equation}\label{BSDEa}
 \begin{cases}
  dV_t=b(t,\VX_t,\BZ_t)dt+\sigma(t,\VX_t)dW_t,\qquad \VX=(V,X),\\
  dX_t=V_tdt,\\
  dY_t=-f(t,\VX_t,\BZ_t)dt+\BZ_tdW_t,\\
  \VX_0=\y,\quad Y_T=g(\VX_T).
 \end{cases}
\end{equation}
\begin{definition}[\bf Decoupling field]\label{decf}
A decoupling field for \eqref{BSDEa} is a classical solution to the Cauchy problem \eqref{Cauchy}
with
  $$F(t,\zz,u,p) = - b(t,\zz,p)p -f(t,\zz,p),$$
and terminal condition $G = g$.
\end{definition}
\begin{remark}\label{r3}
Assume that a decoupling field $u$ for \eqref{BSDEa} exists and that the (forward) SDE
\begin{equation}\label{degSDE}
 \begin{cases}
  dV_t = b(t,\VX_t,\sigma(t,\VX_t)\p_v u(t,\VX_t))dt + \sigma(t,\VX_t)dW_t,\\
  dX_t = V_t dt
 \end{cases}
\end{equation}
admits a solution with initial datum $\VX_0 =\y$. Applying It\^o's formula to the process
$u(t,\VX_t)$ and setting
\begin{equation}\label{defsol}
 Y_t = u(t,\VX_t), \qquad \BZ_t = \sigma(t,\VX_t)\p_v u(t,\VX_t),
\end{equation}
we obtain that the triple $(\VX,Y,\BZ)$ solves \eqref{BSDEa}. Moreover, by \eqref{defsol} and the
fact that $\VX$ admits a density satisfying uniform Gaussian bounds (see, e.g., \cite{MR4660246}),
we have
\begin{equation}\label{defsol2}
 \|\BZ\|_{L^\infty([0,T]\times\Omega)}
 =\|\sigma \p_v u\|_{L^\infty([0,T]\times\R^{2})}.
\end{equation}

\end{remark}
\begin{notation} We denote by
\begin{itemize}
 \item $\mathcal{S}^2_{0,T}$ the space of continuous adapted processes $S$
such that
  $$E\left[\sup_{t\in[0,T]}|S_{t}|^{2}\right]<\infty.$$
 \item $\mathcal{H}^2_{0,T}$ the space of progressively measurable processes $H$ such that
  $$E\left[\int_{0}^{T}|H_t|^{2}dt\right]<\infty.$$
\end{itemize}
\end{notation}
In Theorem \ref{th2} we prove strong well-posedness for \eqref{BSDEa} by extending the classical
{\it four-step scheme} introduced by Ma, Protter, and Yong \cite{MR1262970}. The proof rests on
two key ingredients:
\begin{itemize}
  \item the strong well-posedness of the FBSDE on sufficiently small time intervals, a classical result under standard Lipschitz continuity assumptions;
  \item Theorem \ref{t1}, which ensures the existence of a decoupling field for \eqref{BSDEa}.
\end{itemize}
By combining these ingredients, we construct the solution iteratively and propagate it in time so
as to cover the entire interval $[0,T]$. The argument crucially exploits the intrinsic regularity
properties of the decoupling field, allowing us to extend local well-posedness to arbitrarily long
time horizons and thereby obtain global strong well-posedness.

\begin{assumption}\label{ass6}
There exists positive constants $\k$ and $\bar{\a}\in(0,1]$, such that the following conditions
hold for any $t\in[0,T]$, $\zz,\zz'\in\times\R^{2}$ and $p,p'\in \R$.
\begin{enumerate}
\item \emph{Non-degeneracy and growth conditions:}
\begin{align}
  \k^{-1}\le \s(t,\zz)\le\k,\qquad |b(t,\zz,p)|+|g(\zz)|\lesssim 1,\qquad |f(t,\zz,p)|\lesssim
  1+|p|.
\end{align}
\item \emph{Regularity:} $\s\in\Lip_{0,T}(\R^{2})$, $b,f\in\Lip_{0,T}(\R^{3})$ and
$g\in\Lip(\R^{2})$. Moreover, $\s$ is differentiable in $v$, and $b,f$ are differentiable in $v,p$
and satisfy
\begin{equation}\label{and5}
\begin{split}
  |\p_{v} \s(t,\zz)-\p_{v} \s(t,\zz')|&\lesssim |\zz-\zz'|^{\bar{\a}}_{\K},\\
  |\p_{v} b(t,\zz,p)-\p_{v} b(t,\zz',p')|+|\p_{p} b(t,\zz,p)-\p_{p} b(t,\zz',p')|&\lesssim
  |\zz-\zz'|^{\bar{\a}}_{\K}+|p-p'|^{\bar{\a}},\\
  |\p_{v}f(t,\zz,p)-\p_{v}f(t,\zz',p')|+|\p_{p}f(t,\zz,p)-\p_{p}f(t,\zz',p')|&\lesssim
  |\zz-\zz'|^{\bar{\a}}_{\K}+|p-p'|^{\bar{\a}}.
\end{split}
\end{equation}
\end{enumerate}
\end{assumption}
\begin{theorem}\label{th2}
Let $\y \in L^2(\O,\mathcal{F}_0,P)$. Under Assumption \ref{ass6}, the FBSDE \eqref{BSDEa} admits
\begin{itemize}
  \item[i)] a unique strong solution $(\VX,Y,\BZ) \in \mathcal{S}^2_{0,T} \times
  \mathcal{S}^2_{0,T} \times \mathcal{H}^2_{0,T}$;
  \item[ii)] a decoupling field $u\in\Lip_{0,T}(\R^{2})\cap\C^{3+\a}_{\K,t}$
  for any $\a\in(0,\bar{\a})$ and $t\in[0,T)$.
\end{itemize}
\end{theorem}
\begin{proof}
We begin by proving strong existence. By Theorem \ref{t1}, there exists a decoupling field
$u\in\Lip_{0,T}(\R^{2})$ for \eqref{BSDEa}. Theorem~\ref{t1} also ensures that $u \in
\C^{2+\a}_{\K,t}$ for any $\a \in (0,1)$ and $t\in[0,T)$. However, this regularity is not
sufficient to ensure strong well-posedness of \eqref{degSDE}. The obstruction arises from the
degeneracy of the system. Indeed, the strongest currently available results on regularization by
noise for SDEs of the form \eqref{degSDE}, due to Chaudru de Raynal \cite{MR3606742}, require that
the drift coefficient
  $$
  b(t,v,x,\sigma(t,v,x)\p_v u(t,v,x))$$
be H\"older continuous in $x$ with exponent strictly larger than $\frac{2}{3}$ (and H\"older
continuous in $v$ with some positive exponent). Now, the condition $u\in \C^{2+\a}_{\K,t}$ only
yields $\p_v u\in C^{1+\a}_{\K,t}$, which in turn implies $\p_v u(t,\cdot)\in
C^{\frac{1+\a}{3}}(\R^{2})$ in the standard Euclidean sense and this level of regularity is not
sufficient to satisfy the above requirement.

At this stage we invoke the stronger regularity assumption \eqref{and5} on the coefficients, which
guarantee that, setting $\hat{\ff}(t,z):=\ff(t,z,u(t,z),\p_v u(t,z))$, we have
\begin{equation}\label{ass4}
  \s,\hat{\ff}\in C^{1+\a}_{\K,t},\qquad t\in[0,T),\ \a\in(0,\bar{\a}).
\end{equation}
This yields higher regularity for the solution. Indeed, under condition \eqref{ass4}, the Schauder
estimates for equation $\K u=\hat{F}$ proved in \cite{Luce2023} (see also Proposition~3.9 and
Lemma~3.10 in \cite{MR1826957}) yield $u\in \C^{3+\a}_{\K,t}$ for every $\a\in(0,\bar{\a})$ and
every $t\in[0,T)$. In particular, $\p_v u\in C^{2+\a}_{\K,t},$ and this regularity is enough to
guarantee that \eqref{degSDE} is strongly well posed.

\medskip We next address uniqueness. Let $(\tilde{\VX},\tilde{Y},\tilde{\BZ})$ be a strong solution to
\eqref{BSDEa}. By the uniform-in-$t$ Lipschitz continuity of $u,\sigma$ in $\zz$, and of $b,f,g$
in $(\zz,\bz)$, the standard local theory for FBSDEs (see, e.g., Theorem~1.1 in \cite{MR1901154})
applies. In particular, there exists $\d>0$ sufficiently small such that, for every $s \in
[0,T-\d]$, the FBSDE on the interval $[s,s+\d]$ with initial condition $\VX_s=\tilde{\VX}_s$ and
terminal condition
  $Y_{s+\d}=u(s+\d,\VX_{s+\d})$
admits a unique strong solution $(\VX,Y,\BZ)$ satisfying $(\VX,Y)\in\mathcal{S}^{2}_{s,s+\d}$ and
$\BZ\in\mathcal{H}^{2}_{s,s+\d}$.

We introduce a partition $0=t_0<t_1<\cdots<t_N=T$ such that $t_{k+1}-t_k\le\d$ for all $k$:
arguing by backward induction on $k$, we show that
  $$\tilde{Y}_t = u(t,\tilde{\VX}_t)
  \qquad\hbox{and}\qquad \tilde{\BZ}_t =\sigma(t,\tilde{\VX}_t) \p_v u(t,\tilde{\VX}_t), \qquad
  t\in[t_k,t_{k+1}],$$
for every $k=0,\dots,N-1$.

\smallskip
\noindent {\bf Step 1: the last interval.} 
On the last interval $[t_{N-1},T]$ the restriction of $(\tilde{\VX},\tilde{Y},\tilde{\BZ})$ solves
the local FBSDE with terminal datum $Y_T=u(T,\VX_T)$. By local uniqueness on $[t_{N-1},T]$, this
solution must coincide with the solution constructed via the decoupling field $u$; hence, $$
\tilde{Y}_t = u(t,\tilde{\VX}_t), \qquad \tilde{\BZ}_t = \sigma(t,\tilde{\VX}_t)\p_v
u(t,\tilde{\VX}_t), \qquad t\in[t_{N-1},T]. $$ In particular,
$\tilde{Y}_{t_{N-1}}=u(t_{N-1},\tilde{\VX}_{t_{N-1}})$ a.s.

\smallskip
\noindent {\bf Step 2: induction.} Assume that for some $k\in\{0,\dots,N-2\}$ we already know
$\tilde{Y}_{t_{k+1}}=u(t_{k+1},\tilde{\VX}_{t_{k+1}})$ a.s. Then, on the interval $[t_k,t_{k+1}]$,
the restriction of $(\tilde{\VX},\tilde{Y},\tilde{\BZ})$ solves the local FBSDE with terminal
condition
  $$ Y_{t_{k+1}} = \tilde{Y}_{t_{k+1}} = u(t_{k+1},\tilde{\VX}_{t_{k+1}}). $$
By local uniqueness on $[t_k,t_{k+1}]$, it follows that this restriction must coincide with the
solution associated with the terminal condition $u(t_{k+1},\cdot)$. Hence,
  $$\tilde{Y}_t =u(t,\tilde{\VX}_t), \qquad \tilde{\BZ}_t = \sigma(t,\tilde{\VX}_t)\p_v u(t,\tilde{\VX}_t), \qquad
  t\in[t_k,t_{k+1}], $$
and in particular $\tilde{Y}_{t_k}=u(t_k,\tilde{\VX}_{t_k})$ a.s. This closes the backward
induction.

\smallskip
\noindent We have thus shown that $\tilde{Y}_t=u(t,\tilde{\VX}_t)$ and
$\tilde{\BZ}_t=\sigma(t,\tilde{\VX}_t)\p_v u(t,\tilde{\VX}_t)$ for all $t\in[0,T]$. Plugging this
identity into the forward equation of \eqref{BSDEa}, we see that $\tilde{\VX}$ solves the feedback
SDE \eqref{degSDE}. By strong well-posedness of the above SDE, we obtain that $\tilde{\VX}$ and
$\VX$ are indistinguishable; consequently, also $\tilde{Y}=Y$ in $\mathcal{S}^{2}_{0,T}$ and
$\tilde{\BZ}=\BZ$ in $\mathcal{H}^{2}_{0,T}$. This proves uniqueness for \eqref{BSDEa}.
\end{proof}

\subsection{Frozen optimal control problem}
Given a flow of marginals $\m=(\mu_t)_{t\in[0,T]}\in C([0,T],\mathcal{P}_2(\R^2))$, we consider
the frozen optimal control problem \eqref{frozen_SOC}. With Theorems \ref{t1} and \ref{th2} at our
disposal, the following result follows by adapting the arguments in \cite{MR3752669}, Section
4.4.3.
\begin{theorem}\label{th4}
Fix a flow of marginals $\m\in C([0,T],\mathcal{P}_2(\R^2))$ and an initial datum $\y \in
L^2(\O,\mathcal{F}_0,P)$. Under Assumptions \ref{ass1} and \ref{ass2}, the frozen FBSDE
\eqref{linFBSDE} admits:
\begin{itemize}
  \item[i)] a unique strong solution $(\VX^{\m},Y^{\m},\BZ^{\m}) \in \mathcal{S}^2_{0,T} \times
  \mathcal{S}^2_{0,T} \times \mathcal{H}^2_{0,T}$;
  \item[ii)] a decoupling field $u^{\m}\in\Lip_{0,T}(\R^{2})\cap\C^{3+\a}_{\K,t}$,
  for any $\a\in(0,\bar{\a})$ and any $t\in[0,T)$, which is the classical solution
  of the HJB problem \eqref{HJB}.
\end{itemize}
Moreover, $\VX^{\m}$ is the optimal state process for problem \eqref{frozen_SOC} and
\begin{align}
 a^{\m}_t &:= 
 \hat{a}(t,\VX_t^{\m},\m_t,\p_v u^{\m}(t,\VX^{\m}_t)),\qquad t\in[0,T]
 \end{align}
is the optimal control process. In particular, the optimal cost is
  $$
  J^\m(a^{\m})=E[u^{\m}(0,\eta)].$$
\end{theorem}

\subsection{Conclusion: the fixed point argument}
We conclude the proof of Theorem \ref{th3} by means of a fixed-point argument. Let a flow of
marginals $\m\in C([0,T],\mathcal{P}_2(\R^2))$ be given. We denote by $(\VX^{\m},Y^{\m},\BZ^{\m})$
the solution to the frozen FBSDE \eqref{linFBSDE}, and by $u^{\m}$ the associated decoupling
field. The forward component $\VX^{\m}$ then solves the SDE
\begin{equation}\label{linFSDE2}
 \begin{cases}
  dV^{\m}_t = \b(t,\VX^{\m}_{t},\mu_t)dt
  + \sigma(t,\VX^{\m}_{t},\mu_t)dW_{t},\\
  dX^{\m}_t = V^{\m}_tdt,
 \end{cases}
\end{equation}
with initial condition $\VX^{\m}_{0}=\y$, where
\begin{equation}\label{beta}
  \b(t,z,\m):= b(t,z,\mu,\hat{a}(t,z,\mu,\p_v u^{\m}(t,z))).
\end{equation}

We show that the map
\begin{equation}\label{Phi}
 \Phi:\m \longmapsto \left([\VX^{\mu}_t]\right)_{t\in[0,T]}
\end{equation}
is a contraction with respect to the norm
\begin{align}
  \mathcal{W}_{1,T}(\mu,\nu) &:= \max_{t\in[0,T]}\mathcal{W}_{1}(\mu_{t},\nu_{t})\\ \label{KR}
  &= \max_{t\in[0,T]}\sup_{\|\phi\|_{\Lip(\R^{2})}\le 1}\int_{\R^2} \phi(z)(\mu_t-\nu_t)(dz).
\end{align}
Equipped with $\mathcal{W}_{1,T}$, the space $C([0,T],\mathcal{P}_1(\R^2))$ is a Banach space. As
anticipated in Remark \ref{r10}, rather than working in $C([0,T],\mathcal{P}_2(\R^2))$, we seek a
fixed point of $\Phi$ in the larger space $C([0,T],\mathcal{P}_1(\R^2))$, in order to exploit the
Kantorovich-Rubinstein duality \eqref{KR}. Once a fixed point has been obtained in
$\mathcal{P}_1(\R^2)$, we recover the additional square integrability, namely the finiteness of
the second moment, by standard a priori estimates for SDEs.

We note that the coefficient $\b$ in \eqref{beta} is bounded and satisfies
\begin{equation}\label{crucial}
  |\b(t,z,\m)-\b(t,z',\m')|
  \lesssim |z-z'|+\mathcal{W}_{1}(\m,\m'),\qquad
  z,z'\in\R^{2},\ \m,\m'\in \mathcal{P}_1(\R^2).
\end{equation}
Estimate \eqref{crucial} follows from two ingredients. On the one hand, it is a consequence of
assumptions \eqref{w11}-\eqref{w1} on $b$ and $\hat{a}$. On the other hand, it relies on classical
continuous-dependence results for PDEs, extended to the ultraparabolic setting in
\cite{MR2352998}. More precisely, for each $\m\in\mathcal{P}_1(\R^2)$, the function $u^{\m}$
solves \eqref{HJB}, and its stability with respect to the parameter $\m$ yields the
$\mathcal{W}_{1}$-dependence appearing in \eqref{crucial}.

We now consider the backward Kolmogorov operator $\A_{\m}+\Y$ associated with \eqref{linFSDE2},
where
\begin{equation}\label{oper2}
  \A_{\m}:=\frac{\sigma^2(t,\zz,\mu_t)}{2}\partial_{vv} +\b(t,\zz,\mu_t)\p_{v}
\end{equation}
and $\Y=v\p_{x}+\p_{t}$ is the drift vector field in \eqref{Y}. Under the current assumptions on
$\s$ and $b$, the existence of a fundamental solution to $\A_{\m}+\Y$, i.e. the transition density
of $\VX^{\m}$ in \eqref{linFSDE2}, together with Gaussian bounds, was established in
\cite{MR4660246}.
\begin{theorem}\label{t11}
The operator $\A_{\m}+\Y$ admits a fundamental solution $p^{\m}=p^{\m}(t,z;s,z')$ satisfying the
following Gaussian estimates:
\begin{align}
  \G^{\l_-}(t,\zz;s,\zz')\lesssim p^{\m}(t,\zz;s,\zz')&\lesssim\G^{\l_+}(t,\zz;s,\zz'),\\ \label{es1}
  |\p_{v}p^{\m}(t,\zz;s,\zz')|&\lesssim\frac{1}{\sqrt{s-t}}\G^{\l_+}(t,\zz;s,\zz'),\\ \label{es2}
  |\p_{vv}p^{\m}(t,\zz;s,\zz')|&\lesssim\frac{1}{s-t}\G^{\l_+}(t,\zz;s,\zz'),\qquad 0\le t<s\le T,\
  \zz,\zz'\in\R^{2},
\end{align}
where $\l_{\pm}$ are positive structural constants and $\G^\l$ denotes the fundamental solution of
$\frac{\l}{2}\p_{vv}+v\p_{x}+\p_{t}$, namely
\begin{equation}\label{Gau}
  \G^\l(t,v,x;s,v',x') =
  \frac{\sqrt{3}}{\pi\l(s-t)^{2}}\exp\left(-\frac{2(v'-v)^{2}}{\l(s-t)}+
  \frac{6(v'-v)(x'-x-v(s-t))}{\l(s-t)^{2}}-\frac{6(x'-x-v(s-t))^{2}}{\l(s-t)^{3}}\right).
\end{equation}
\end{theorem}

Next, we introduce the {\it push-forward} and {\it pull-back} operators acting on a distribution
$\th\in\mathcal{P}(\R^{2})$: $$\vec{P}^{\m}_{t,s}\th(z')
 :=\int_{\R^{2}}p^\m(t,z;s,z')\th(d z),\qquad
 \cev{P}^{\m}_{t,s}\n(z):=\int_{\R^{2}} p^\m(t,z;s,z')\th(d z'),$$
for $0\le t<s\le T$ and $z,z'\in\R^{2}$. We also introduce the forward and backward estimators of
the distance between two flows of distributions $\m,\n\in C([0,T];\mathcal{P}(\R^2))$:
\begin{align}\label{e23}
 \vec{I}_{t_1,t_2}^{\m,\n}(\th,\phi)
 &=\int_{\R^{2}}d z\phi(z)(\vec{P}^{\m}_{t_1,t_2}-\vec{P}^{\n}_{t_1,t_2})\th(z),
 \\
 \label{e30}
 \cev{I}_{t_1,t_2}^{\m,\n}(\th,\phi)
 &=\int_{\R^{2}}\th(d z)\int_{t_1}^{t_2}d t\cev{P}^{\m}_{t_1,t}
 \left(\A_{\m}-\A_{\n}\right)\cev{P}^{\n}_{t,t_2}\phi(z),
\end{align}
for any $0\le t_1<t_2\le T$, $\th\in\mathcal{P}(\R^2)$, and $\phi\in\Lip(\R^2)$. The forward and
backward estimators are related by the following duality formula, proved in \cite{MR4848654}:
\begin{equation}\label{e100}
  \vec{I}_{t_1,t_2}^{\m,\n}(\th,\phi)=\cev{I}_{t_1,t_2}^{\m,\n}(\th,\phi).
\end{equation}
Note that $\vec{P}^{\m}_{0,s}[\y]$ is the density of the marginal law $[\VX^{\m}_{s}]$, and, by
\eqref{KR}, we have
\begin{equation}\label{e22}
  \mathcal{W}_{1,T}([\VX^{\m}],[\VX^{\n}])=\max_{s\in[0,T]}\sup_{\|\phi\|_{\Lip(\R^{2})}\le 1}
  \vec{I}_{0,s}^{\m,\n}([\y],\phi).
\end{equation}

We are now in a position to prove that $\Phi$ in \eqref{Phi} is a contraction, at least for $T$
sufficiently small. Indeed, by \eqref{e100} we have
\begin{align}
 \left|\vec{I}_{0,s}^{\m,\n}([\y],\phi)\right|
 &=\left|\cev{I}_{0,s}^{\m,\n}([\y],\phi)\right|\\
 &\le
 \int_{0}^{s}\|\cev{P}^{\m}_{0,t}(\A_{\m}-\A_{\n})\cev{P}^{\n}_{t,s}\phi\|_{L^{\infty}}d t\\
 &\le \int_{0}^{s}\|(\A_{\m}-\A_{\n})\cev{P}^{\n}_{t,s}\phi\|_{L^{\infty}}d t\lesssim
\intertext{(by the Lipschitz continuity of $\s$ and $\b$ with respect to the measure argument; see
\eqref{w1} and \eqref{crucial})}
 &\lesssim \mathcal{W}_{1,T}(\m,\n)\int_{0}^{s}
  \left(\|\p_{vv}\cev{P}^{\n}_{t,s}\phi\|_{L^{\infty}}+
  \|\p_{v}\cev{P}^{\n}_{t,s}\phi\|_{L^{\infty}}\right)d t \lesssim
\intertext{(by the Gaussian estimates \eqref{es1}-\eqref{es2} and the standard potential estimates
for $\Gamma^{\l}$; see, e.g., Proposition 5.3 in \cite{DiFrancescoPascucci2})}\label{e7}
  &\lesssim\mathcal{W}_{1,T}(\m,\n)\int_{0}^{s}\frac{1}{\sqrt{s-t}}d t
  =2\sqrt{s}\mathcal{W}_{1,T}(\m,\n).
\end{align}
Thus, by \eqref{e22}, we obtain
\begin{equation}
 \mathcal{W}_{1,T}([\VX^{\m}],[\VX^{\n}])\le C \sqrt{T}\mathcal{W}_{1,T}(\m,\n)
\end{equation}
for some structural constant $C>0$ independent of $T$. This yields the contraction property for
$T$ sufficiently small. Therefore, $\Phi$ in \eqref{Phi} admits a fixed point
$\bar{\m}=[\VX^{\bar{\m}}]\in C([0,T],\mathcal{P}_1(\R^2))$. Moreover, since $\VX^{\bar{\m}}$
solves \eqref{linFSDE2} with coefficients of linear growth, classical SDE estimates imply that, in
fact, $\bar{\m}=[\VX^{\bar{\m}}]\in C([0,T],\mathcal{P}_2(\R^2))$. A continuation argument extends
the result to any finite time horizon $T$, thereby completing the proof.

\bibliographystyle{acm}
\bibliography{Bibtex-Final}
\end{document}